\documentclass[11pt,a4paper]{article}
\title{\bf The It\^o SDEs and Fokker--Planck equations with Osgood and Sobolev coefficients}

\author{Dejun Luo\footnote{Email: luodj@amss.ac.cn. The author is supported by the National Natural Science Foundation of China (11571347), the Seven Main Directions (Y129161ZZ1) and the Special Talent Program of the Academy of Mathematics and Systems Science, Chinese Academy of Sciences.} \vspace{2mm} \\
{\footnotesize Key Laboratory of Random Complex Structures and Data Sciences, } \\
{\footnotesize Academy of Mathematics and Systems Science, Chinese Academy of Sciences, Beijing 100190, China}}
\date{}
\usepackage{amssymb,amsmath,amsfonts,amsthm,array,color}

\def\R{\mathbb{R}}
\def\E{\mathbb{E}}
\def\P{\mathbb{P}}

\def\F{\mathcal{F}}
\def\L{\mathcal{L}}

\def\ch{{\bf 1}}

\def\div{\textup{div}}
\def\d{\textup{d}}
\def\supp{\textup{supp}}

\def\<{\langle}
\def\>{\rangle}

\def\XXint#1#2#3{{\setbox0=\hbox{$#1{#2#3}{\int}$}
\vcenter{\hbox{$#2#3$}}\kern-.5\wd0}}

\begin{document}

\maketitle
\makeatletter 
\renewcommand\theequation{\thesection.\arabic{equation}}
\@addtoreset{equation}{section}
\makeatother 

\newtheoremstyle{newthm}
 {3pt}
 {3pt}
 {\itshape}
 {}
 {}
 {\textbf{.}}
 {.5em}
 {\thmname{\textbf{#1}}\thmnumber{ \textbf{#2}}\thmnote{ {\textbf{(#3)}}}}

\theoremstyle{newthm}

\newtheorem{theorem}{Theorem}[section]
\newtheorem{lemma}[theorem]{Lemma}       
\newtheorem{corollary}[theorem]{Corollary}
\newtheorem{proposition}[theorem]{Proposition}
\newtheorem{remark}[theorem]{Remark}
\newtheorem{example}[theorem]{Example}
\newtheorem{definition}[theorem]{Definition}

\vspace{-5mm}

\begin{abstract}
We study the degenerated It\^o SDE on $\R^d$ whose drift coefficient only fulfills a mixed Osgood and Sobolev regularity. Under suitable assumptions on the gradient of the diffusion coefficient and on the divergence of the drift coefficient, we prove the existence and uniqueness of generalized stochastic flows associated to such equations. We also prove the uniqueness of solutions to the corresponding Fokker--Planck equation by using the probabilistic method.
\end{abstract}


\textbf{Keywords:} Stochastic differential equation, Osgood and Sobolev condition, DiPerna--Lions theory, Fokker--Planck equation, stochastic flow

\textbf{MSC 2010:} Primary 60H10; secondary 35Q84

\section{Introduction}

Let $\sigma:\R^d\to \mathcal M_{d\times m}$ be a matrix-valued function and $b:\R^d\to \R^d$ a vector field on $\R^d$. We consider the following stochastic differential equation (SDE for short)
  \begin{equation}\label{Ito-SDE}
  \d X_t=\sigma(X_t)\,\d B_t+b(X_t)\,\d t,
  \end{equation}
where $B_t$ is an $m$-dimensional standard Brownian motion defined on some probability space $(\Omega,\F,\P)$. According to the classical theory on SDE, if $\sigma$ and $b$ are globally Lipschitz continuous, then the equation \eqref{Ito-SDE} generates a unique stochastic flow $X_t$ of homeomorphisms on $\R^d$. In the pioneer work \cite{Malliavin}, Malliavin constructed the canonical Brownian motion on the diffeomorphism group of the circle $S^1$, see \cite{Fang} for a more detailed construction and \cite{AiraultRen} for the explicit modulus of H\"older continuity of the flow associated to the canonical Brownian motion. Since then, there have been intensive studies on SDEs with non-Lipschitz coefficients, see for instance \cite{FangZhang03, RenZhang, Fang04, FangZhang, RenZhang05, Zhang05}. In particular, the existence of a unique strong solution to \eqref{Ito-SDE} was proved in \cite{FangZhang} under the general Osgood type condition. In the case that the coefficients have the log-Lipschitz continuity, X. Zhang \cite{Zhang05} established the homeomorphic property of the stochastic flow by following Kunita's approach (cf. \cite[Section 4.5]{Kunita90}).

If $\sigma\equiv 0$, then \eqref{Ito-SDE} reduces to an ordinary differential equation (abbreviated as ODE):
  \begin{equation}\label{ODE}
  \frac{\d X_t}{\d t}=b(X_t).
  \end{equation}
In recent years, the study of ODEs with weakly differentiable coefficients attracted lots of attentions, see e.g. \cite{DiPernaLions89, Ambrosio04, CiprianoCruzeiro05} for the finite dimensional case, \cite{AmbrosioFigalli09, FangLuo10, Trevisan} for extensions on the Wiener space, \cite{Dumas, Zhang11, FangLiLuo} for studies on Riemannian manifolds.  The readers can find a survey of some of these results in \cite{Ambrosio11}. In these works (except \cite{Zhang11}), the existence and uniqueness of quasi-invariant flows $X_t$ generated by \eqref{ODE} are deduced from the well-posedness of the corresponding transport equation or the continuity equation. By making use of the pointwise characterization of Sobolev functions in $W^{1,p}_{loc}(\R^d)$ (see \eqref{Sobolev.1} below), Crippa and de Lellis \cite{CrippadeLellis} are able to give a direct construction of the flow $X_t$. This method was developed in \cite{BouchutCrippa13} to show the well posedness of \eqref{ODE} when the gradient of $b$ is given by some singular integral. Following this direct method, there are also studies on the SDE \eqref{Ito-SDE} with coefficients having Sobolev regularity \cite{Zhang10, FangLuoThalmaier, Zhang13, Luo15b}. Regarding the corresponding PDE, Le Bris and Lions studied in \cite{LeBrisLions08} the Fokker--Planck type equations with Sobolev coefficients; using Ambrosio's commutator estimate for BV vector fields, their results was slightly extended in \cite{Luo13} to the case where the drift coefficient has only BV regularity. Based on a representation formula for the solutions to Fokker--Planck equations (see \cite[Theorem 2.6]{Figalli}), the uniqueness was established in \cite{RocknerZhang10, Luo14} when the coefficients are bounded and have weak spatial regularity.

This paper is a continuation of the work \cite{Luo15}, where the authors propose a unified framework for ODEs \eqref{ODE} under the mixed Osgood and Sobolev conditions on the coefficient $b$. The spaces consist of these kind of functions have been studied intensively in the past two decades, even in the case where the underlying space is the general metric measure space (motivated by the pioneer work \cite{Hajlasz}); see \cite{Ivanishko} for the equivalence of different Sobolev spaces and \cite{IK} for the compactness of embeddings of Sobolev type. Our purpose is to extend the main results in \cite{Luo15} to the case of the It\^o SDE \eqref{Ito-SDE}. Compared to equation  \eqref{ODE}, a big difference in the stochastic setting is that the estimate of the Radon--Nikodym density involves the gradient of the diffusion coefficient $\sigma$ (see e.g. Lemma \ref{sect-2-lem-4} in the current paper), which implies that $\sigma$ naturally has some Sobolev regularity. Therefore, the mixed Osgood and Sobolev regularity can only be imposed  on the drift coefficient $b$ in \eqref{Ito-SDE}. The first main result of this paper also extends \cite[Theorem 2.3]{Luo15b}, in which the drift coefficient $b$ is required to be in the first order Sobolev space. We remark that when the equation \eqref{Ito-SDE} has non-degenerate diffusion coefficient $\sigma$, the existence of a unique strong solution can be proved under quite weak conditions on the drift $b$, see \cite{KrylovRockner, FF13b} for integrability conditions and \cite{FGP10a, Wang} for weak continuity conditions.

This paper is organized as follows. In Section 2, we first give the meaning of the generalized stochastic flow associated to \eqref{Ito-SDE}, then we state the mixed Osgood and Sobolev condition $(\mathbf{H}_q)\, (q\geq 1)$ and provide an example of functions satisfying such condition. The main results consist of three theorems: the first one (Theorem \ref{1-thm-0}) allows the diffusion coefficient $\sigma$ to be in some Sobolev space but requires that $b$ fulfills $(\mathbf{H}_q)$ with $q>1$; in Theorem \ref{1-thm-1}, we assume $(\mathbf{H}_1)$ on the drift $b$ and that $\sigma$ is smooth, which is mainly due to the estimate of the Radon--Nikodym density of stochastic flows; the last main result (Theorem \ref{4-thm}) proves the uniqueness of the related Fokker--Planck equation under the Sobolev regularity on $\sigma$ and the mixed Osgood and Sobolev condition $(\mathbf{H}_1)$ on $b$. The subsequent three sections are devoted to the proofs of these theorems, respectively.

\section{Preparations and main results}

First, we present the precise definition of the generalized stochastic flow (cf. \cite[Definition 5.1]{FangLuoThalmaier} and \cite[Definition 2.1]{Zhang13}) associated to \eqref{Ito-SDE}. As usual, the space of continuous functions on $\R^d$ is denoted by $C([0,T],\R^d)$. Let $\mu$ be a locally finite measure on $\R^d$ which is absolutely continuous with respect to the Lebesgue measure $\L^d$. For a measurable map $\varphi:\R^d\to\R^d$, we write $\varphi_\#\mu =\mu \circ\varphi^{-1}$ for the push-forward of $\mu$ by $\varphi$ (also called the distribution of $\varphi$ under $\mu$).

\begin{definition}\label{sect-2-def}
A measurable map $X\colon\Omega\times\R^d\to C([0,T],\R^d)$ is called a $\mu$-a.e. stochastic flow associated to the It\^{o} SDE \eqref{Ito-SDE} if
\begin{enumerate}
\item[\rm(i)] for each $t\in [0,T]$ and almost all $x\in\R^d$, $\omega\to X_t(\omega,x)$ is measurable with respect to $\F_t$, i.e., the natural filtration generated by the Brownian motion $\{B_s\colon s\leq t\}$;

\item[\rm(ii)] there exists a nonnegative function $K:[0,T] \times\Omega \times\R^d\to \R_+$ such that for each $t\in [0,T]$, $(X_t(\omega,\cdot))_\#\mu =K_t\mu$;

\item[\rm(iii)] for $(\P\times\mu)$-a.e. $(\omega,x)$,
  \begin{equation*}
  \int_0^T|\sigma (X_s(\omega,x))|^2\,\d s+ \int_0^T|b (X_s(\omega,x))|\,\d s<+\infty;
  \end{equation*}
\item[\rm(iv)] for $\mu$-a.e. $x\in\R^d$, the integral equation below holds almost surely:
  \begin{equation*}
  X_t(\omega,x)=x+\int_0^t \sigma (X_s(\omega,x))\,\d B_s+\int_0^t b (X_s(\omega,x))\,\d s,\quad \mbox{for all } t\in [0,T].
  \end{equation*}
\end{enumerate}
\end{definition}

Throughout this paper, we fix a nondecreasing function $\rho\in C^1(\R_+,\R_+)$ which satisfies $\rho(0)=0$ and $\int_{0+}\frac{\d s}{\rho(s)}=\infty$. Without loss of generality, we assume $\rho(s)\geq s$ for all $s\geq0$. Typical examples for the function $\rho(s)$ are $s,\, s\log\frac1s,\, s(\log\frac1s)(\log\log\frac1s),\cdots$. Although the latter two functions are only well defined on a small neighborhood of the origin, we can extend them to the whole positive half line by piecing them together with linear functions. For example,
  \begin{equation}\label{Osgood-Sobolev.3}
  \rho(s)=\begin{cases}
  s\log\frac1s, & s\in [0,e^{-2}];\\
  s+e^{-2}, & s\in (e^{-2},\infty).
  \end{cases}
  \end{equation}
Alternatively, one can directly use $\rho(s)=s\log\big(\frac1s +e\big)$. For any $\delta>0$, we define the following auxiliary function
  $$\psi_\delta(\xi)=\int_0^\xi \frac{\d s}{\rho(s) +\delta},\quad \xi>0.$$
Note that $\lim_{\delta \downarrow 0} \psi_\delta(\xi)=\infty$ for all $\xi>0$. Moreover,
  \begin{equation}\label{auxi-funct}
  \psi'_\delta(\xi)=\frac1{\rho(\xi) +\delta}>0,\quad \psi''_\delta(\xi)=-\frac{\rho'(\xi)}{(\rho(\xi) +\delta)^2}\leq 0.
  \end{equation}
This property shows that $\psi_\delta$ is a concave function for any $\delta>0$. If $\rho(s) \equiv s$ for all $s\geq 0$, then $\psi_\delta(\xi)= \log\big(1+\frac{\xi}\delta\big)$ which is the auxiliary function used in the previous works \cite{CrippadeLellis, Zhang10, FangLuoThalmaier, Luo15b}.

Let $q\geq 1$ be fixed. We are now ready to introduce the following hypothesis.

\begin{enumerate}
\item[$(\mathbf{H}_q)$] For any $R>0$, there exist a nonnegative function $g_R\in L^q_{loc}(\R^d)$ and negligible subset $N$, such that for all $x,y\notin N$ with $|x-y|\leq R$, one has
  \begin{equation}\label{Osgood-Sobolev.1}
  |\<x-y,b(x)-b(y)\>|\leq \big(g_R(x) + g_R(y)\big) \rho(|x-y|^2).
  \end{equation}
\end{enumerate}

Here is an example of functions satisfying $(\mathbf{H}_q)$.

\begin{example}\label{1-exa-1} \rm
Take $b_1\in W^{1,q}_{loc}(\R^d,\R^d)$. If $q=1$, we require further that $|\nabla b_1| \in (L^1\log L^1)_{loc}$. It is well known that
  \begin{equation}\label{Sobolev.1}
  |b_1(x)-b_1(y)|\leq C_d\big(M_R |\nabla b_1|(x)+M_R |\nabla b_1|(y)\big)|x-y| \quad \mbox{for a.e. } x,y \mbox{ with } |x-y|\leq R,
  \end{equation}
where $C_d$ is a dimensional constant and $M_R |\nabla b_1|$ is the local maximal function of $|\nabla b_1|$:
  $$M_R |\nabla b_1|(x)=\sup_{0<r\leq R}\frac1{\L^d(B_r)}\int_{B_r} |\nabla b_1|(x+y)\,\d y,\quad x\in\R^d,$$
in which $B_r$ is the ball centered at origin with radius $r$. Moreover, for any $\lambda >0$, if $q=1$, then
  $$\int_{B_\lambda} M_R |\nabla b_1|(x)\,\d x\leq C_d \int_{B_{R+\lambda}} |\nabla b_1(x)| \log(1+|\nabla b_1(x)| )\,\d x;$$
while if $q>1$, then
  \begin{equation}\label{Sobolev.2}
  \int_{B_\lambda} \big(M_R |\nabla b_1|(x)\big)^q\,\d x\leq C_{d,q} \int_{B_{R+\lambda}} |\nabla b_1(x)|^q \,\d x.
  \end{equation}

Next, let $b_2(x)=(V(x_1),\cdots, V(x_d))$ with
  $$V(t)=\sum_{k=1}^\infty \frac{|\sin kt|}{k^2},\quad t\in\R.$$
Then by \cite[(2.12)]{FangZhang}, we have
  $$|b_2(x)-b_2(y)|\leq C_0 d \rho(|x-y|) \quad \mbox{for all } x,y\in\R^d,$$
where $\rho(s)$ is given in \eqref{Osgood-Sobolev.3}. Now it is easy to show that the vector field $b=b_1+b_2$ satisfies $(\mathbf{H}_q)$ with $g_R=C'_d(1+M_R|\nabla b_1|)$.
\end{example}

To simplify notations, we write $\bar b(x)=\frac{b(x)}{1+|x|}$ and $\bar\sigma(x)=\frac{\sigma(x)}{1+|x|}$ for $x\in \R^d$. Our first main result extends \cite[Theorem 2.3]{Luo15b}.

\begin{theorem}\label{1-thm-0}
Let $q>1$ and $\d\mu(x)=(1+|x|^2)^{-q -(d+1)/2}\,\d x$. Assume that
\begin{itemize}
\item[\rm(i)] $\sigma\in W^{1,2}_{loc}$ and $b$ satisfies $(\mathbf{H}_q)$ and the distributional divergence $\div(b)$ exists;
\item[\rm(ii)]  for any $p>0$, one has
  \begin{equation}\label{1-thm-0.1}
  \int_{\R^d} \exp\big\{ p\big[ (\div(b))^-+|\bar b| + |\bar\sigma|^2 +|\nabla\sigma|^2 \big]\big\}\,\d\mu<+\infty.
  \end{equation}
\end{itemize}
Then there exists a unique $\mu$-a.e. stochastic flow $X_t$ associated to the It\^o SDE \eqref{Ito-SDE}. Moreover, the Radon--Nikodym density $K_t:=\frac{\d[(X_t)_\#\mu]}{\d\mu}$ of the flow $X_t$ belongs to $L^\infty([0,T],L^p(\P\times\mu))$ for any $p>1$.
\end{theorem}

\begin{remark}\label{1-remark-0} \rm
The condition \eqref{1-thm-0.1} has the following consequences:
\begin{itemize}
\item[(a)] By the Sobolev embedding theorem, the diffusion coefficient $\sigma$ is H\"older continuous.
\item[(b)] As noted in \cite[Remark 2.2(ii)]{Luo15b}, the condition \eqref{1-thm-0.1} implies $\bar\sigma,\bar b\in L^p(\mu)$ for any $p>1$. Moreover, by the choice of the measure $\mu$, if $p$ is sufficiently big, then $\int_{\R^d}(1+|x|)^{2qp/(p-1)} \,\d\mu<+\infty$. H\"older's inequality yields
  $$\int_{\R^d}|\sigma|^{2q}\,\d\mu \leq \bigg[\int_{\R^d}|\bar\sigma|^{2qp}\,\d\mu \bigg]^{1/p} \bigg[\int_{\R^d}(1+|x|)^{2qp/(p-1)} \,\d\mu\bigg]^{(p-1)/p}<+\infty.$$
Thus $\sigma\in L^{2q}(\mu)$. In the same way we have $b\in L^{2q}(\mu)$.
\end{itemize}
\end{remark}

Theorem \ref{1-thm-0} will be proved in Section 3. Under the assumption $(\mathbf{H}_1)$, we need stronger conditions on the diffusion coefficient $\sigma$.

\begin{theorem}\label{1-thm-1}
Assume that
\begin{itemize}
\item[\rm(i)] the diffusion coefficient $\sigma\in C_b^2(\R^d,\mathcal M_{d\times m})$, i.e. it is bounded with bounded spatial derivatives up to order two;
\item[\rm(ii)] the drift coefficient $b\in L^1_{loc}$ satisfies $(\mathbf{H}_1)$ and the distributional divergence $\div(b)$ exists such that
  \begin{equation}\label{1-thm-1.2}
  [\div(b)]^- \in L^\infty(\R^d),\quad \bar b\in L^\infty(B_r^c)\quad \mbox{for some } r>0.
  \end{equation}
\end{itemize}
Then there exists a unique $\L^d$-a.e. stochastic flow $X_t$ generated by It\^o SDE \eqref{Ito-SDE}.
\end{theorem}

Recall that $\bar b(x)=\frac{b(x)}{1+|x|}$ and $B_r^c$ is the complement of the ball $B_r$. By \eqref{1-thm-1.2}, $b$ can be locally unbounded. This result has two main differences from \cite[Theorem 2.2]{Zhang13}: (1) the assumption on $\sigma$ is stronger here, but it is much easier to be checked; (2) the $W^{1,1}_{loc}$-regularity of the drift $b$ is replaced by $(\mathbf{H}_1)$.

The It\^o SDE \eqref{Ito-SDE} is closely related to the Fokker--Planck equation
  \begin{equation}\label{FPE}
  \partial_t \mu_t=L^\ast \mu_t,\quad \mu|_{t=0}=\mu_0,
  \end{equation}
where $L^\ast$ is the adjoint operator of $L$ defined as
  \begin{equation}\label{sect-4.1}
  L\varphi(x)=\frac12\sum_{i,j=1}^d a^{ij}(x)\partial_{ij}\varphi(x) +\sum_{i=1}^d b^i(x)\partial_i\varphi(x),\quad \varphi\in C_c^\infty(\R^d).
  \end{equation}
Here $a^{ij}(x)=\sum_{k=1}^m \sigma^{ik}(x)\sigma^{jk}(x),\, \partial_i\varphi(x)=\frac{\partial\varphi}{\partial x_i}(x)$ and $\partial_{ij}\varphi(x) =\frac{\partial^2\varphi}{\partial x_i\partial x_j}(x),\,1\leq i,j\leq d$. This equation is understood as follows: for any $\varphi\in C_c^\infty(\R^d)$,
  $$\frac{\d}{\d t}\int_{\R^d}\varphi(x)\,\d\mu_t(x)=\int_{\R^d} L \varphi(x)\,\d\mu_t(x),$$
where the initial condition means that $\mu_t$ weakly$\ast$ converges to $\mu_0$ as $t$ tends to 0. If $\mu_t$ is absolutely continuous with respect to the Lebesgue measure with the  density function $u_t$ for all $t\in[0,T]$, then the density function $u_t$ solves the PDE below in the weak sense:
  \begin{equation}\label{FPE-1}
  \partial_t u_t=L^\ast u_t,\quad u|_{t=0}=u_0.
  \end{equation}
The next result is a direct consequence of the It\^o formula.

\begin{proposition}\label{4-prop}
Assume the conditions of Theorem \ref{1-thm-1}. Let $(X_t)_{0\leq t\leq T}$ be the generalized stochastic flow associated to \eqref{Ito-SDE} and $K_t$ the Radon--Nikodym density of $(X_t)_\# \L^d$ with respect to $\L^d$. Then $u_t(x):=\E K_t(x)$ solves the Fokker--Planck equation \eqref{FPE-1} with $u_0=1$.
\end{proposition}

Our main purpose is to show the uniqueness of the Fokker--Planck equation \eqref{FPE-1} in a suitable space, following the ideas in \cite{RocknerZhang10, Luo14}. This approach is a probabilistic one, based on Figalli's formula (see \cite[Theorem 2.6]{Figalli}) which represents the solution to the Fokker--Planck equation \eqref{FPE} in terms of the martingale solution corresponding to the It\^o SDE \eqref{Ito-SDE}. Note that this representation formula is only valid for bounded coefficients $\sigma$ and $b$ (see \cite[p. 149]{Bogachev15} for related discussions). Nevertheless, it enables us to treat the degenerate case. We remark that there are many works dealing with various kind of non-degenerate equations
\cite{Bogachev07, Bogachev13, Bogachev15}. The recent book \cite{Bogachev15-2} presents a comprehensive study on Fokker--Planck--Kolmogorov equations.

We introduce the following condition on the diffusion coefficient:
\begin{enumerate}
\item[$(\mathbf{H}_\sigma)$] For any $R>0$, there exist a nonnegative function $\tilde g_R\in L^1_{loc}(\R^d)$ and negligible subset $N$, such that for all $x,y\notin N$ with $|x-y|\leq R$, one has
  \begin{equation}\label{Osgood-Sobolev-sigma}
  \|\sigma(x)-\sigma(y)\|^2\leq \big(\tilde g_R(x) + \tilde g_R(y)\big) \rho(|x-y|^2).
  \end{equation}
\end{enumerate}
By \eqref{Sobolev.1} and \eqref{Sobolev.2} $(q=2)$, it is clear that if $\sigma$ belongs to the Sobolev space $W^{1,2}_{loc}$, then the above condition holds with $\tilde g_R= 2C_d^2 (M_R|\nabla\sigma|)^2$ and $\rho(s)=s$. Here is the last main result of this work which slightly generalizes \cite[Theorem 1.1]{RocknerZhang10}.

\begin{theorem}[Uniqueness of Fokker--Planck equations]\label{4-thm}
Assume that the coefficients $\sigma$ and $b$ are essentially bounded. Moreover, the hypotheses $(\mathbf{H}_\sigma)$ and $(\mathbf{H}_1)$ hold for $\sigma$ and $b$ respectively. Then for any given probability density function $f\in \mathcal B_b(\R^d)$, there is at most one weak solution $u_t$ to the Fokker--Planck equation \eqref{FPE-1} in the class $L^\infty\big([0,T], L^1\cap L^\infty(\R^d)\big)$ with $u_0=f$.
\end{theorem}

\section{Proof of Theorem \ref{1-thm-0}}

First we prove an a-priori moment estimate on the solution flow $X_t$ to \eqref{Ito-SDE}. This improves the result presented in \cite[Lemma 2.4]{Luo15b}.

\begin{lemma}[Moment estimate]\label{sect-2-lem-1}
Assume that $q>1$ and $\sigma,\, b\in L^{2q}(\mu)$. Let $X_t$ be a $\mu$-a.e. stochastic flow associated to It\^o SDE \eqref{Ito-SDE}, and $K_t$ the Radon--Nikodym density with respect to $\mu$. If for some $q_1\in [1,q)$, one has
  \begin{equation}\label{sect-2-lem-1.0}
  \Lambda_{q,q_1,T}:=\sup_{0\leq t\leq T}\|K_t\|_{L^{q/(q-q_1)}(\P\times\mu)}<+\infty,
  \end{equation}
then we have
  $$\int_{\R^d}\E\sup_{0\leq t\leq T}|X_t(x)|^{2q_1} \,\d\mu(x) \leq C_{d,q_1}+ C_{q_1,T} \Lambda_{q,q_1,T}\big(\|\sigma\|_{L^{2q}(\mu)}^{2q_1} +\|b\|_{L^{2q}(\mu)}^{2q_1}\big).$$
\end{lemma}

\begin{proof}
For any $R>0$, define the stopping time $\tau_R(x)=\inf\{t>0: |X_t(x)|\geq R\}$. To simplify notations, we shall omit the space variable $x$ in $X_s(x)$ and $\tau_R(x)$. The It\^o formula yields
  $$|X_{t\wedge \tau_R}|^2=|x|^2 +2\int_0^{t\wedge \tau_R}\<X_s, \sigma(X_s)\,\d B_s\> +2\int_0^{t\wedge \tau_R} \<X_s, b(X_s)\>\,\d s + \int_0^{t\wedge \tau_R} \|\sigma(X_s)\|^2\,\d s.$$
Hence there is $C_{q_1}>0$ such that
  \begin{equation}\label{sect-2-lem-1.1}
  \begin{split}
  |X_{t\wedge \tau_R}|^{2q_1}&\leq C_{q_1}\bigg[|x|^{2q_1} +\bigg|\int_0^{t\wedge \tau_R}\<X_s, \sigma(X_s)\,\d B_s\>\bigg|^{q_1} +\bigg|\int_0^{t\wedge \tau_R} \<X_s, b(X_s)\>\,\d s\bigg|^{q_1}\\
  &\hskip38pt + \bigg(\int_0^{t\wedge \tau_R} \|\sigma(X_s)\|^2\,\d s\bigg)^{q_1}\bigg].
  \end{split}
  \end{equation}

By Burkholder's inequality,
  \begin{align*}
  \E\sup_{t\leq T}\bigg|\int_0^{t\wedge \tau_R} \<X_s, \sigma(X_s)\,\d B_s\>\bigg|^{q_1}
  &\leq C'_{q_1}\,\E\bigg[\bigg(\int_0^{T\wedge \tau_R}|\sigma(X_s)^\ast X_s|^2\,\d s\bigg)^{q_1/2}\bigg]\\
  &\leq C'_{q_1} \,\E\bigg[\bigg(\sup_{s\leq T\wedge\tau_R} |X_s|^{q_1}\bigg) \bigg(\int_0^{T\wedge \tau_R} \|\sigma(X_s)\|^2\,\d s \bigg)^{q_1/2}\bigg].
  \end{align*}
Cauchy's inequality leads to
  \begin{equation}\label{sect-2-lem-1.2}
  \begin{split}
  \E\sup_{t\leq T}\bigg|\int_0^{t\wedge \tau_R} \<X_s, \sigma(X_s)\,\d B_s\>\bigg|^{q_1}
  & \leq \frac1{3C_{q_1}}\, \E\sup_{s\leq T\wedge\tau_R}|X_s|^{2q_1} + C'_{q_1,T}\,\E \int_0^{T\wedge \tau_R}\|\sigma(X_s)\|^{2q_1}\,\d s.
  \end{split}
  \end{equation}
Next,
  \begin{align*}
  \E\sup_{t\leq T} \bigg|\int_0^{t\wedge \tau_R} \<X_s, b(X_s)\>\,\d s\bigg|^{q_1}
  &\leq \E \bigg(\int_0^{T\wedge \tau_R} |\<X_s, b(X_s)\>|\,\d s\bigg)^{q_1}\\
  &\leq \E\bigg[\bigg(\sup_{s\leq T\wedge\tau_R} |X_s|^{q_1}\bigg)\bigg(\int_0^{T\wedge \tau_R} |b(X_s)| \,\d s\bigg)^{q_1} \bigg].
  \end{align*}
Again by Cauchy's inequality,
  \begin{equation}\label{sect-2-lem-1.3}
  \begin{split}
  \E\sup_{t\leq T} \bigg|\int_0^{t\wedge \tau_R} \<X_s, b(X_s)\>\,\d s\bigg|^{q_1}
  &\leq \frac1{3C_{q_1}}\, \E \sup_{s\leq T\wedge\tau_R}|X_s|^{2q_1} + C''_{q_1} \E\bigg(\int_0^{T\wedge \tau_R} |b(X_s)| \,\d s\bigg)^{2q_1}\\
  &\leq \frac1{3C_{q_1}}\, \E \sup_{s\leq T\wedge\tau_R}|X_s|^{2q_1} + C''_{q_1,T} \E\int_0^{T\wedge \tau_R} |b(X_s)|^{2q_1} \,\d s.
  \end{split}
  \end{equation}
Finally,
  \begin{equation}\label{sect-2-lem-1.4}
  \E\sup_{t\leq T} \bigg(\int_0^{t\wedge \tau_R} \|\sigma(X_s)\|^2\,\d s\bigg)^{q_1} \leq \tilde C_{q_1,T} \E\int_0^{T\wedge \tau_R} \|\sigma(X_s)\|^{2q_1} \,\d s.
  \end{equation}

Note that $\sup_{t\leq T}|X_{t\wedge \tau_R}|^{2q_1} =\sup_{t\leq T\wedge\tau_R}|X_t|^{2q_1}$. Combining \eqref{sect-2-lem-1.1}--\eqref{sect-2-lem-1.4}, we obtain
  $$\E\sup_{t\leq T\wedge\tau_R}|X_t|^{2q_1}\leq 3C_{q_1}|x|^{2q_1} + C_{q_1,T} \E\int_0^{T} \ch_{\{\tau_R >s\}} \big(\|\sigma(X_s)\|^{2q_1} + |b(X_s)|^{2q_1}\big) \,\d s.$$
Integrating both sides on $\R^d$ with respect to $\mu$ yields
  \begin{align*}
  \int_{\R^d}\E\sup_{t\leq T\wedge\tau_R}|X_t|^{2q_1} \,\d\mu
  &\leq C_{d,q_1}+ C_{q_1,T} \int_0^T \E \int_{\R^d} \ch_{\{\tau_R >s\}} \big(\|\sigma(X_s)\|^{2q_1} + |b(X_s)|^{2q_1}\big) \,\d\mu\d s\\
  &\leq C_{d,q_1}+ C_{q_1,T} \int_0^T \E \int_{\R^d} \big(\|\sigma(y)\|^{2q_1} + |b(y)|^{2q_1}\big)K_s(y) \,\d \mu(y)\d s,
  \end{align*}
where $C_{d,q_1}=3C_{q_1}\int_{\R^d} |x|^{2q_1}\,\d\mu(x)<+\infty$. Using H\"older's inequality and \eqref{sect-2-lem-1.0}, we get
  \begin{align*}
  \int_{\R^d}\E\sup_{t\leq T\wedge\tau_R}|X_t|^{2q_1} \,\d\mu
  &\leq C_{d,q_1}+ C_{q_1,T} \big(\|\sigma\|_{L^{2q}(\mu)}^{2q_1} +\|b\|_{L^{2q}(\mu)}^{2q_1}\big) T\sup_{s\leq T} \|K_s\|_{L^{q/(q-q_1)}(\P\times \mu)} \\
  &\leq C_{d,q_1}+ C_{q_1,T} T\Lambda_{q,q_1,T}\big(\|\sigma\|_{L^{2q}(\mu)}^{2q_1} +\|b\|_{L^{2q}(\mu)}^{2q_1}\big).
  \end{align*}
Since the right hand side is independent of $R>0$, Fatou's lemma allows us to let $R\to\infty$ to get the desired estimate.
\end{proof}

Next we provide a stability estimate which will play an important role in the proof of Theorem \ref{1-thm-0}. Recall the definition of the auxiliary function $\psi_\delta$ for $\delta>0$ in Section 2. We denote by $\|\cdot \|_{\infty, T}$ the uniform norm in $C([0,T],\R^d)$.

\begin{lemma}[Stability estimate]\label{sect-2-lem-2}
Assume that $\sigma,\tilde\sigma\in L^{2q}_{loc}(\R^d, \mathcal{M}_{d\times m})$ and $b,\tilde b\in L^q_{loc}(\R^d, \R^d)$. Moreover, $\sigma$ and $b$ satisfy the condition {\rm (i)} in Theorem \ref{1-thm-0}. Let $X_t$ (resp. $\tilde X_t$) be the stochastic flow associated to the It\^o SDE \eqref{Ito-SDE} with coefficients $\sigma$ and $b$ (resp. $\tilde\sigma$ and $\tilde b$). Denote by $K_t$ (resp. $\tilde K_t$) the Radon--Nikodym density of $X_t$ (resp. $\tilde X_t$) with respect to $\mu$. Assume that
  \begin{equation}\label{sect-2-lem-2.1}
  \Lambda_{q,T}:=\sup_{0\leq t\leq T}\big[\| K_t\|_{L^{q/(q-1)}(\P\times\mu)} \vee \| \tilde K_t\|_{L^{q/(q-1)}(\P\times\mu)}\big]<+\infty.
  \end{equation}
Then for any $\delta>0$,
  \begin{align*}
  &\hskip14pt \E\int_{G_R}\psi_\delta\big(\|X-\tilde X\|^2_{\infty,T}\big)\,\d \mu  \\
  & \leq C_{d,R,T}\big(\Lambda_{q,T}\wedge\sqrt{\Lambda_{q,T}} \,\big) \bigg[\frac1{\delta}\|\sigma- \tilde\sigma \|_{L^{2q}(B_R)}^2 +\frac1{\sqrt\delta}\Big(\|\sigma-\tilde\sigma\|_{L^{2q}(B_R)}+ \|b-\tilde b\|_{L^q(B_R)} \Big)\\
  &\hskip126pt +\|\nabla\sigma\|_{L^{2q}(B_{3R})}^2 +\|\nabla\sigma\|_{L^{2q}(B_{3R})} +\|g_{2R}\|_{L^q(B_R)}\bigg],
  \end{align*}
where $G_R(\omega):=\big\{x\in \R^d:\|X_\cdot(\omega,x)\|_{\infty,T} \vee\|\tilde X_\cdot(\omega,x)\|_{\infty,T} \leq R\big\}$, and $L^q(B_R)$ is the usual function space with respect to the Lebesgue measure on the ball $B_R=\{x\in\R^d: |x|\leq R\}$.
\end{lemma}

\begin{proof}
We follow the line of arguments of \cite[Theorem 5.2]{FangLuoThalmaier}. Denote by $Z_t=X_t-\tilde X_t$ and $\xi_t=|Z_t|^2$ where we omit the space variable $x$. It\^o's formula leads to
  $$\d\xi_t=2\<Z_t,(\sigma(X_t)- \tilde\sigma(\tilde X_t))\,\d B_t\>+2\<Z_t,b(X_t)-\tilde b(\tilde X_t)\>\,\d t+\|\sigma(X_t)- \tilde\sigma(\tilde X_t)\|^2\,\d t.$$
Applying again the It\^o formula and by \eqref{auxi-funct}, we obtain
  \begin{equation}\label{sect-2-lem-2.2}
  \begin{split}
  \d\psi_\delta(\xi_t) &\leq 2\frac{\<Z_t,(\sigma(X_t)- \tilde\sigma(\tilde X_t))\,\d B_t\>}{\rho(\xi_t)+\delta} +2\frac{\<Z_t,b(X_t)-\tilde b(\tilde X_t)\>}{\rho(\xi_t)+\delta}\,\d t + \frac{\|\sigma(X_t)- \tilde\sigma(\tilde X_t) \|^2} {\rho(\xi_t)+\delta}\,\d t \\
  &=: \d I_1(t)+ \d I_2(t)+ \d I_3(t).
  \end{split}
  \end{equation}

We first estimate the term $I_1(t)$. Define the stopping time $\tau_R(x)=\inf\{t>0: |X_t(x)|\vee |\tilde X_t(x)|> R\}$. Note that a.s. $G_R\subset B_R$. Thus
  \begin{equation}\label{sect-2-lem-2.3}
  \E\int_{G_R}\sup_{0\leq t\leq T}|I_1(t)|\,\d\mu \leq \E\int_{B_R} \sup_{0\leq t\leq T\wedge\tau_R} |I_1(t)|\,\d\mu.
  \end{equation}
By Burkholder's inequality, we have
  \begin{align*}
  \E\sup_{0\leq t\leq T\wedge\tau_R} |I_1(t)|^2&\leq 4\, \E\int_0^{T\wedge\tau_R} \frac{\big|(\sigma(X_s)-\tilde\sigma(\tilde X_s))^\ast Z_s\big|^2}{(\rho(\xi_s)+\delta)^2}\,\d s\\
  &\leq 4\, \E\int_0^{T\wedge\tau_R} \frac{\|\sigma(X_s)- \tilde\sigma(\tilde X_s)\|^2}{\xi_s+\delta}\,\d s
  \end{align*}
since $\rho(\xi)\geq \xi\geq 0$. Cauchy's inequality leads to
  \begin{align*}
  \int_{B_R}\E \sup_{0\leq t\leq T\wedge\tau_R} |I_1(t)| \,\d \mu
  &\leq \int_{B_R} 2\bigg[\E\int_0^{T\wedge\tau_R} \frac{\|\sigma(X_s)- \tilde\sigma(\tilde X_s)\|^2} {\xi_s+\delta}\,\d s\bigg]^{1/2}\,\d\mu\\
  &\leq C_{d,R} \bigg[\int_{B_R} \E\int_0^{T\wedge\tau_R} \frac{\|\sigma(X_s)- \tilde\sigma(\tilde X_s)\|^2}{\xi_s+\delta}\,\d s \d \mu\bigg]^{1/2}.
  \end{align*}
where $C_{d,R}=2(\mu(B_R))^{1/2}$. It is clear that
  \begin{equation}\label{sect-2-lem-2.4}
  \begin{split}
  \int_{B_R}\E \sup_{0\leq t\leq T\wedge\tau_R} |I_1(t)| \,\d\mu
  &\leq C_{d,R} \bigg[\int_{B_R} \E\int_0^{T\wedge\tau_R} \frac{\|\sigma(X_s)- \sigma(\tilde X_s)\|^2}{\xi_s+\delta}\,\d s \d \mu\bigg]^{1/2}\\
  &\quad + C_{d,R} \bigg[\int_{B_R} \E\int_0^{T\wedge\tau_R} \frac{\|\sigma(\tilde X_s)- \tilde\sigma(\tilde X_s)\|^2}{\xi_s+\delta}\,\d s \d \mu\bigg]^{1/2}\\
  &=: I_{1,1} + I_{1,2}.
  \end{split}
  \end{equation}
We have by H\"older's inequality and \eqref{sect-2-lem-2.1} that
  \begin{equation}\label{sect-2-lem-2.5}
  \begin{split}
  I_{1,2}&\leq \frac{C_{d,R}}{\sqrt\delta} \bigg[\int_0^{T} \E\int_{B_R} \ch_{\{\tau_R> s\}} \|\sigma(\tilde X_s)- \tilde\sigma(\tilde X_s) \|^2\,\d \mu\d s \bigg]^{1/2}\\
  &\leq \frac{C_{d,R}}{\sqrt\delta} \bigg[\int_0^{T} \E\int_{B_R} \|\sigma(y)- \tilde \sigma(y)\|^2 \tilde K_s(y)\,\d \mu(y)\d s\bigg]^{1/2}\\
  &\leq \frac{C_{d,R}}{\sqrt\delta} \sqrt{T \Lambda_{q,T}}\, \|\sigma - \tilde \sigma\|_{L^{2q}(B_R)},
  \end{split}
  \end{equation}
where in the last step we used the fact that the measure $\mu$ is smaller than the Lebesgue measure. Next, by \eqref{Sobolev.1},
  \begin{align*}
  I_{1,1} &= C_{d,R} \bigg[\int_0^{T} \E \int_{B_R} \ch_{\{\tau_R > s\}}\frac{\|\sigma(X_s)- \sigma(\tilde X_s)\|^2}{\xi_s+\delta}\,\d \mu\d s \bigg]^{1/2}\\
  &\leq C'_{d,R} \bigg[\int_0^{T} \E \int_{B_R} \ch_{\{\tau_R > s\}} \big[M_{2R}|\nabla\sigma|(X_s)+ M_{2R}|\nabla\sigma|(\tilde X_s) \big]^2 \, \d \mu\d s\bigg]^{1/2}\\
  &\leq \sqrt{2}\, C'_{d,R} \bigg[\int_0^{T} \E \int_{B_R} \big[M_{2R}|\nabla\sigma|(y)\big]^2 \big(K_s(y)+ \tilde K_s(y)\big) \, \d \mu(y)\d s\bigg]^{1/2}.
  \end{align*}
H\"older's inequality leads to
  \begin{align*}
  I_{1,1} &\leq  \sqrt{2}\,  C'_{d,R} \bigg[\int_0^{T} 2\Lambda_{q,T} \bigg(\int_{B_R} \big[M_{2R}|\nabla\sigma|(y)\big]^{2q} \, \d \mu(y)\bigg)^{1/q} \d s \bigg]^{1/2}\\
  &\leq C''_{d,R} \sqrt{T \Lambda_{q,T}} \|\nabla\sigma\|_{L^{2q}(B_{3R})},
  \end{align*}
where in the second step we have used the maximal inequality \eqref{Sobolev.2}. Combining this inequality with \eqref{sect-2-lem-2.3}--\eqref{sect-2-lem-2.5}, we arrive at
  \begin{equation}\label{sect-2-lem-2.6}
  \E \int_{G_R} \sup_{0\leq t\leq T} |I_1(t)| \,\d\mu \leq \bar C_{d,R,T} \sqrt{\Lambda_{q,T}}\bigg( \frac1{\sqrt\delta} \|\sigma - \tilde \sigma\|_{L^{2q}(B_R)} + \|\nabla\sigma\|_{L^{2q}(B_{3R})}\bigg).
  \end{equation}

Next, we treat the second term $I_2(t)$. We have
  \begin{align*}
  \E\int_{G_R}\sup_{0\leq t\leq T}|I_2(t)|\,\d\mu
  & \leq \E\int_{B_R}\sup_{0\leq t\leq T\wedge\tau_R}|I_2(t)|\,\d\mu\\
  &\leq 2\, \E\int_{B_R}\int_0^{T\wedge\tau_R} \frac{|\<Z_s,b(X_s)-\tilde b(\tilde X_s)\>|}{\rho(\xi_s)+\delta}\,\d s\d\mu\\
  &=2 \int_0^T \E\int_{B_R}{\bf 1}_{\{\tau_R>s\}} \frac{|\<Z_s,b(X_s)-\tilde b(\tilde X_s)\>|}{\rho(\xi_s)+\delta}\,\d\mu\d s.
  \end{align*}
Note that $|\<Z_s,b(X_s)-\tilde b(\tilde X_s)\>|\leq |\<Z_s,b(X_s)-b(\tilde X_s)\>|+|\<Z_s,b(\tilde X_s)-\tilde b(\tilde X_s)\>|$. We deduce from $(\mathbf{H}_q)$ that
  \begin{align*}
  \E\int_{B_R}{\bf 1}_{\{\tau_R>s\}} \frac{|\<Z_s,b(X_s)-b(\tilde X_s)\>|}{\rho(\xi_s)+\delta}\,\d \mu &\leq \E\int_{B_R}{\bf 1}_{\{\tau_R>s\}} \big[g_{2R}(X_s)+g_{2R}(\tilde X_s)\big]\,\d \mu\\
  &\leq \E\int_{B_R} g_{2R}(y) \big[K_s(y)+\tilde K_s(y)\big]\,\d \mu,
  \end{align*}
which, combined with \eqref{sect-2-lem-2.1}, gives us
  $$ \int_0^T \E\int_{B_R}{\bf 1}_{\{\tau_R>s\}} \frac{|\<Z_s,b(X_s)-b(\tilde X_s)\>|}{\rho(\xi_s)+\delta}\,\d \mu\d s \leq 2T\Lambda_{q,T} \|g_{2R}\|_{L^q(B_R)}.$$
Moreover, since $\rho(s)\geq s$ for all $s\geq 0$,
  \begin{align*}
  \int_0^T \E\int_{B_R}{\bf 1}_{\{\tau_R>s\}} \frac{|\<Z_s,b(\tilde X_s)-\tilde b(\tilde X_s)\>|}{\rho(\xi_s) +\delta}\,\d \mu\d s&\leq \int_0^T \E\int_{B_R}{\bf 1}_{\{\tau_R>s\}} \frac{|b(\tilde X_s)-\tilde b(\tilde X_s)|}{\sqrt{ \xi_s+\delta}}\,\d \mu\d s\\
  &\leq \frac1{\sqrt{\delta}} \int_0^T \E\int_{B_R} |b(y)-\tilde b(y)| \tilde K_s(y)\,\d \mu(y)\d s\\
  &\leq \frac{T \Lambda_{q,T}}{\sqrt{\delta}}\|b-\tilde b\|_{L^q(B_R)}.
  \end{align*}
Consequently,
  \begin{equation}\label{sect-2-lem-2.7}
  \begin{split}
  \E\int_{G_R}\sup_{0\leq t\leq T}|I_2(t)|\,\d\mu &\leq 2T \Lambda_{q,T} \bigg(2\|g_{2R}\|_{L^q(B_R)} + \frac1{\sqrt{\delta}}\|b-\tilde b\|_{L^q(B_R)}\bigg).
  \end{split}
  \end{equation}

Finally, similar to the treatment of the terms on the right hand side of \eqref{sect-2-lem-2.4}, we have
  \begin{align*}
  \E\int_{G_R}\sup_{0\leq t\leq T}|I_3(t)|\,\d \mu &\leq \E\int_{B_R} \int_0^{T\wedge\tau_R} \frac{\|\sigma(X_s)- \tilde\sigma(\tilde X_s)\|^2}{\rho(\xi_s)+\delta}\,\d s\d \mu\\
  &\leq \bar C_{d,R,T} \Lambda_{q,T} \bigg( \frac1{\delta} \|\sigma - \tilde \sigma\|_{L^{2q}(B_R)}^2 + \|\nabla\sigma\|_{L^{2q}(B_{3R})}^2 \bigg).
  \end{align*}
Combining this estimate with \eqref{sect-2-lem-2.6} and \eqref{sect-2-lem-2.7}, we obtain the desired result.
\end{proof}

The uniqueness part of Theorem \ref{1-thm-0} follows directly from Lemma \ref{sect-2-lem-2}.

\begin{corollary}[Uniqueness]\label{sect-2-cor}
Under the conditions of Theorem \ref{1-thm-0}, there exists at most one $\mu$-a.e. stochastic flow associated to the It\^o SDE \eqref{Ito-SDE}.
\end{corollary}

\begin{proof}
Assume there are two stochastic flows $(X_t)_{0\leq t\leq T}$ and $(\tilde X_t)_{0\leq t\leq T}$ associated to the It\^o SDE \eqref{Ito-SDE}. For any $R>0$, applying Lemma \ref{sect-2-lem-2} with $\tilde\sigma=\sigma$ and $\tilde b=b$ leads to
  \begin{equation}\label{sect-2-cor-1}
  \E\int_{G_R}\psi_\delta\big(\|X-\tilde X\|^2_{\infty,T}\big)\,\d \mu \leq \tilde C_{d,R,T} <+\infty,
  \end{equation}
where $\tilde C_{d,R,T}$ depends on $\Lambda_{q,T}$ and $\|\nabla\sigma\|_{L^{2q}(B_{3R})},\, \|g_{2R}\|_{L^q(B_R)}$. For any $\eta\in (0,1)$, let $\Theta_\eta(\omega):=\{x\in \R^d: \|X_\cdot(\omega,x)- \tilde X_\cdot(\omega,x)\|_{\infty,T}\leq \eta\}$. Then by Lemma \ref{sect-2-lem-1} with $q_1=1$ and \eqref{sect-2-cor-1},
  \begin{align*}
  \E\int_{\R^d}\big(1\wedge \|X-\tilde X\|_{\infty,T}^2 \big)\,\d\mu &\leq \E\big[\mu(G_R^c)\big] + \E\int_{G_R} \big(1\wedge \|X-\tilde X\|_{\infty,T}^2 \big)\,\d\mu\\
  &\leq \frac{C_{d,T}}{R^2} +\eta^2 \mu(\R^d) +\E\int_{G_R \cap \Theta_\eta^c}\big(1\wedge \|X-\tilde X\|_{\infty,T}^2 \big)\,\d \mu\\
  &\leq \frac{C_{d,T}}{R^2} +\eta^2 \mu(\R^d) + \frac{\tilde C_{d,R,T}}{\psi_\delta(\eta^2)} .
  \end{align*}
We first let $\delta\downarrow 0$, then let $R \uparrow \infty$ and $\eta\downarrow 0$ to get
  $$\E\int_{\R^d}\big(1\wedge \|X-\tilde X\|_{\infty,T}^2 \big)\,\d \mu =0,$$
which yields the uniqueness of stochastic flows associated to \eqref{Ito-SDE}.
\end{proof}

To prove the existence part of Theorem \ref{1-thm-0}, we shall regularize the coefficients $\sigma$ and $b$ as usual. Let $\chi\in C_c^\infty(\R^d,\R_+)$ be such that $\int_{\R^d}\chi\,\d x=1$ and its support $\supp(\chi)\subset B_1$. For $n\geq 1$, define $\chi_n(x)=n^d\chi(nx)$ for all $x\in\R^d$. Next choose $\phi\in C_c^\infty(\R^d,[0,1])$ which satisfies $\phi|_{B_1} \equiv 1$ and $\supp(\phi)\subset B_2$. Set $\phi_n(x)=\phi(x/n)$ for all $x\in\R^d$ and $n\geq1$. Now we define
  \begin{equation}\label{2-convolution}
  \sigma_n=(\sigma\ast\chi_n)\,\phi_n \quad \mbox{and} \quad b_n=(b\ast\chi_n)\,\phi_n.
  \end{equation}
Then for every $n\geq1$, the functions $\sigma_n$ and $b_n$ are smooth with compact supports. Consider the following It\^o SDE:
  \begin{equation}\label{2-smooth-SDE}
  \d X^n_t=\sigma_n(X^n_t)\,\d B_t +b_n(X^n_t)\,\d t,\quad X^n_0=x.
  \end{equation}
This equation has a unique strong solution which gives rise to a stochastic flow of diffeomorphisms on $\R^d$. Denote by $K^n_t$ the Radon--Nikodym density of $(X^n_t)_\#\mu$ with respect to $\mu$. The following uniform estimate on the density functions was established in \cite[Lemma 2.6]{Luo15b}.

\begin{lemma}[Density estimate]\label{sect-2-lem-4}
For any $p>1$, there are two positive constants $C_{1,p},C_{2,p}>0$ such that
  \begin{equation}\label{sect-2-lem-4.1}
  \begin{split}
  &\hskip14pt \sup_{n\geq 1} \sup_{0\leq t\leq T}\|K^n_t\|_{L^p(\P\times\mu)}\\
  &\leq C_{1,p}\bigg(\int_{\R^d}\exp\big[C_{2,p}T\big([\div(b)]^- +|\bar b| +|\nabla\sigma|^2 +|\bar\sigma|^2\big)\big] \d \mu \bigg)^{\frac1{p(p+1)}}<\infty.
  \end{split}
  \end{equation}
\end{lemma}

Let $q_1\in (1,q)$ and $\Lambda_{q_1,T}$ be the quantity on the right hand side of \eqref{sect-2-lem-4.1} with $p=q/(q-q_1)$, which is finite by \eqref{1-thm-0.1}. Then
  \begin{equation}\label{sect-2.1}
  \sup_{n\geq1}\sup_{0\leq t\leq T}\|K^n_t\|_{L^{q/(q-q_1)}(\P\times\mu)} \leq \Lambda_{q_1,T}<+\infty.
  \end{equation}
Thanks to the moment estimate in Lemma \ref{sect-2-lem-1}, we can improve \cite[Proposition 2.7]{Luo15b} by showing that the sequence of stochastic flows $(X^n_t)_{n\geq 1}$ generated by \eqref{2-smooth-SDE} are convergent in $L^2\big(\Omega\times\R^d, C([0,T],\R^d)\big)$.

\begin{proposition}\label{sect-2-prop-1}
Assume the conditions of Theorem \ref{1-thm-0}. There exists a random field $X:\Omega\times\R^d\to C([0,T],\R^d)$ such that
  $$\lim_{n\to\infty} \E\int_{\R^d} \|X^n -X\|_{\infty,T}^{2} \,\d\mu=0.$$
\end{proposition}

\begin{proof}
The proof follows the line of \cite[Proposition 2.7]{Luo15b}. Applying Lemma \ref{sect-2-lem-1}, we have
  \begin{equation}\label{sect-2-prop-1.0}
  \E\int_{\R^d} \|X^n_\cdot(x)\|_{\infty,T}^{2q_1} \,\d\mu(x) \leq C_{d,q_1}+ C_{q_1,T} \Lambda_{q_1,T}\big(\|\sigma_n\|_{L^{2q}(\mu)}^{2q_1} +\|b_n\|_{L^{2q}(\mu)}^{2q_1}\big),
  \end{equation}
where $C_{d,q_1}$ and $C_{q_1,T}$  are positive constants independent on $n$. We have $|\sigma_n|\leq |\sigma|\ast\chi_n$. By Jensen's inequality,
  $$\|\sigma_n\|_{L^{2q}(\mu)}^{2q}\leq  \int_{\R^d}\big(|\sigma|^{2q} \ast\chi_n\big)(x)\,\d\mu(x) =\int_{\R^d}\chi_n(y) \,\d y\int_{\R^d} \frac{|\sigma(x-y)|^{2q}}{(1+|x|^2)^{q+(d+1)/2}} \,\d x.$$
For $n\geq 2$ and $|y|\leq 1/n$, one has $|x-y|^2\leq 2|x|^2+ 2|y|^2\leq 2|x|^2 +1/2$, hence
  $$1+|x|^2\geq \frac12 |x-y|^2+\frac34\geq \frac12(1+|x-y|^2).$$
Consequently, for $n\geq 2$,
  \begin{equation*}
  \|\sigma_n\|_{L^{2q}(\mu)}^{2q} \leq 2^{q+(d+1)/2} \int_{\R^d}\chi_n(y) \,\d y \int_{\R^d}\frac{|\sigma(x-y)|^{2q}}{(1+|x-y|^2)^{q+(d+1)/2}}\,\d x
  =2^{q+(d+1)/2} \|\sigma\|_{L^{2q}(\mu)}^{2q}.
  \end{equation*}
The same estimate holds for $\|b_n\|_{L^{2q}(\mu)}$. Combining these discussions with \eqref{sect-2-prop-1.0}, we obtain
  \begin{equation}\label{sect-2-prop-1.0.0}
  \sup_{n\geq 2}\E\int_{\R^d} \|X^n_\cdot(x)\|_{\infty,T}^{2q_1} \,\d\mu(x) \leq \hat C<+\infty,
  \end{equation}
where $\hat C$ depends on $d,\, q_1,\, \|\sigma\|_{L^{2q}(\mu)}$ and $\|b\|_{L^{2q}(\mu)}$.

For any $x\in B_R$, we deduce from the definition that $b_n(x)=(b\ast \chi_n)(x)$ for all $n>R$. By \eqref{Osgood-Sobolev.1}, it is clear that one has
  $$|\<x-y,b_n(x)-b_n(y)\>|\leq \big[g^n_{2R}(x) + g^n_{2R}(y)\big] \rho(|x-y|^2)\quad \mbox{for all } x,y\in B_R,$$
where $g^n_{2R}= g_{2R}\ast \chi_n$. Note that $q_1\in (1,q)$, H\"older's inequality yields
  \begin{align}\label{sect-2-prop-1.2}
  \sup_{n\geq1}\sup_{0\leq t\leq T}\|K^n_t\|_{L^{q/(q-1)}(\P\times\mu)}\leq \mu(\R^d)^{(q_1-1)/q}\Lambda_{q_1,T},
  \end{align}
where $\Lambda_{q_1,T}$ is the quantity on the right hand side of \eqref{sect-2-lem-4.1} with $p=q/(q-q_1)$. For any $n\geq1$, we denote by $G_R^{n}$ the level set of the flow $X^n_t$ on the interval $[0,T]$:
  $$G_R^n(\omega)=\{x\in \R^d:\|X^n_\cdot(\omega,x)\|_{\infty,T}\leq R\}.$$
Applying Lemma \ref{sect-2-lem-2} to the flows $X^n_t$ and $X^l_t$ gives us
  \begin{equation}\label{sect-2-prop-1.4}
  \begin{split}
  &\hskip14pt \E\int_{G_R^n\cap G_R^l} \psi_\delta\big(\|X^n- X^l\|^2_{\infty,T}\big)\d\mu\\
  &\leq C_{d,R,T}\Lambda_{q_1,T} \bigg[\frac1{\delta}\|\sigma_n- \sigma_l \|_{L^{2q}(B_R)}^2 +\frac1{\sqrt\delta}\Big(\|\sigma_n- \sigma_l\|_{L^{2q}(B_R)}+ \|b_n- b_l\|_{L^q(B_R)} \Big)\\
  &\hskip75pt +\|\nabla\sigma_n\|_{L^{2q}(B_{3R})}^2 +\|\nabla\sigma_n\|_{L^{2q}(B_{3R})} +\|g^n_{2R}\|_{L^q(B_R)}\bigg].
  \end{split}
  \end{equation}
By the definition of $\sigma_n$, we have
  $$|\nabla \sigma_n|\leq |\nabla \sigma|\ast\chi_n +C\frac{|\sigma\ast\chi_n|}{1+|x|} \leq |\nabla \sigma|\ast\chi_n+ 2C|\bar \sigma|\ast \chi_n.$$
From this we can show that
  \begin{equation}\label{sect-2-prop-1.4.5}
  \|\nabla \sigma_n\|_{L^{2q}(B_{3R})}\leq C_q \big(\|\nabla \sigma\|_{L^{2q}(B_{3R+1})}  +\|\bar \sigma\|_{L^{2q}(B_{3R+1})}\big).
  \end{equation}
Moreover, $\|g^n_{2R}\|_{L^q(B_R)}\leq \|g_{2R}\|_{L^q(B_{R+1})}$. Hence for any $n\geq 1$,
  $$ \|\nabla\sigma_n\|_{L^{2q}(B_{3R})}^2 +\|\nabla\sigma_n\|_{L^{2q}(B_{3R})} +\|g^n_{2R}\|_{L^q(B_R)}  \leq C'_{d,q,R}<+\infty.$$
Now we define
  $$\delta_{n,l}=\big(\|\sigma_n-\sigma_l\|_{L^{2q}(B_{R})}+\|b_n- b_l\|_{L^{q}(B_{R})}\big)^2$$
which tends to $0$ as $n,l\to+\infty$. Taking $\delta=\delta_{n,l}$
in \eqref{sect-2-prop-1.4}, we obtain that for any $n,l\geq1$,
  \begin{equation}\label{sect-2-prop-1.5}
  \E\int_{G_R^n\cap G_R^l} \psi_{\delta_{n,l}} \big(\|X^n- X^l\|^2_{\infty,T}\big)\d\mu \leq C_{T,d,q,q_1,R}<+\infty.
  \end{equation}

The moment estimate \eqref{sect-2-prop-1.0.0}  implies
  \begin{equation*}
  \E\int_{(G_R^n\cap G_R^l)^c}\big(1\wedge\|X^n-X^l\|_{\infty,T}^2 \big)\,\d\mu  \leq \E\big\{\mu\big[(G_R^n\cap G_R^l)^c\big]\big\}  \leq \frac{4\hat C}{R^{2q_1}}.
  \end{equation*}
For $\eta\in (0,1)$, set
  $$\Sigma^{n,l}_\eta (\omega)=\big\{x\in \R^n:\|X^n(\omega, x)-X^l(\omega, x)\|_{\infty,T}\leq \eta\big\}.$$
We have
  \begin{equation*}
  \begin{split}
  \E\int_{G_R^n\cap G_R^l}\! \big(1\wedge\|X^n-X^l\|_{\infty,T}^2 \big)\,\d\mu
  &\leq \eta\,\mu(\R^d) +\E\int_{(G_R^n\cap G_R^l)\setminus \Sigma^{k,l}_\eta}\big(1\wedge\|X^n-X^l\|_{\infty,T}^2 \big)\,\d \mu\\
  &\leq \eta\,\mu(\R^d) +\frac1{\psi_{\delta_{n,l}} (\eta^2)} \E\int_{G_R^n\cap G_R^l}\! \psi_{\delta_{n,l}} \big(\|X^n- X^l\|^2_{\infty,T}\big)\,\d\mu\\
  &\leq \eta\,\mu(\R^d) +\frac{C_{T,d,q,q_1,R}}{\psi_{\delta_{n,l}} (\eta^2)},
  \end{split}
  \end{equation*}
where the last inequality follows from \eqref{sect-2-prop-1.5}. Combining the above two estimates, we obtain
  $$\E\int_{\R^d} \big(1\wedge\|X^n-X^l\|_{\infty,T}^2 \big)\,\d\mu\leq \frac{4\hat C}{R^{2q_1}} + \eta\,\mu(\R^d) +\frac{C_{T,d,q,q_1,R}}{\psi_{\delta_{n,l}} (\eta^2)}.$$
Recall that $\delta_{n,l}$ tends to 0 as $n,l\to\infty$, hence, first letting $n,l\to \infty$, and then $R\to+\infty$, $\eta\to0$, we arrive at
  $$\lim_{n,l\to+\infty}\E\int_{\R^d} \big(1\wedge\|X^n-X^l\|_{\infty,T}^2 \big)\,\d\mu=0.$$
Taking into account \eqref{sect-2-prop-1.0.0}, we deduce that
  $$\lim_{n,l\to+\infty}\E\int_{\R^d} \|X^n-X^l\|_{\infty,T}^2 \,\d\mu=0.$$
This immediately implies the desired result.
\end{proof}

At this stage, we can use \eqref{sect-2-prop-1.2} and follow the arguments of \cite[Theorem 3.4]{FangLuoThalmaier} to show that

\begin{proposition}\label{sect-2-prop-2}
Let $X:\Omega\times\R^d\to C([0,T],\R^d)$ be the random field obtained in Proposition \ref{sect-2-prop-1}. For all $t\in[0,T]$, there exists $K_t:\Omega\times\R^d\to\R_+$ such that $(X_t)_\#\mu=K_t \mu$. Moreover, $\sup_{0\leq t\leq T} \|K_t\|_{L^{q/(q-1)}(\P\times\mu)}\leq \mu(\R^d)^{(q_1-1)/q} \Lambda_{q_1, T}$.
\end{proposition}

\begin{remark}
It follows from Proposition \ref{sect-2-prop-2} and H\"older's inequality that
  $$\aligned \int_0^{T} \E \int_{\R^d} |\sigma(X_s(x))|^2\,\d\mu(x)\d s&= \int_0^{T} \E \int_{\R^d} |\sigma(y)|^2 K_s(y)\,\d\mu(y)\d s\\
  &\leq T\|\sigma\|_{L^{2q}(\mu)}^2 \sup_{0\leq s\leq T} \|K_s\|_{L^{q/(q-1)}(\P\times \mu)}\\
  &\leq T\|\sigma\|_{L^{2q}(\mu)}^2 \mu(\R^d)^{(q_1-1)/q} \Lambda_{q_1,T} <\infty. \endaligned $$
Fubini's theorem implies $\E \int_0^{T} |\sigma(X_s(x))|^2\,\d s<\infty$ for $\mu$-a.e. $x\in\R^d$. Therefore, the process $[0,T] \ni t \to \int_0^t \sigma(X_s(x))\,\d B_s$ is a square integrable martingale.
\end{remark}

To show that $(X_t)_{0\leq t\leq T}$ solves the It\^o SDE \eqref{Ito-SDE}, we need the following preparation.

\begin{lemma}\label{2-lem-5} We have
  $$\lim_{n\to \infty}\E\int_{\R^d} \sup_{0\leq t\leq T}  \bigg|\int_0^t\big[\sigma_n(X^n_s)-\sigma(X_s)\big]\d B_s\bigg|^2 \d \mu(x)=0$$
and
  $$\lim_{n\to \infty}\E\int_{\R^d} \sup_{0\leq t\leq T} \bigg|\int_0^t\big[b_n(X^n_s)-b(X_s)\big]\d s\bigg|^2 \d \mu(x)=0.  $$
\end{lemma}

\begin{proof}
It is clear that
  $$\lim_{n\to\infty}\|\sigma_n-\sigma\|_{L^{2q}(\mu)}=0  \quad\mbox{and}\quad \lim_{n\to\infty}\|b_n-b\|_{L^{2q}(\mu)}=0.$$
Combining these limits with Propositions \ref{sect-2-prop-1} and \ref{sect-2-prop-2}, and making use of the uniform density estimate \eqref{sect-2-prop-1.2}, we can finish the proof as in \cite[Proposition 4.1]{FangLuoThalmaier}.
\end{proof}

For any $n\geq1$, we rewrite the equation \eqref{2-smooth-SDE} in the integral form:
  \begin{equation}\label{smooth-SDE-integral}
  X^n_t(x)=x+\int_0^t\sigma_n(X^n_s)\,\d B_s +\int_0^t b_n(X^n_s)\,\d s.
  \end{equation}
When $n\to +\infty$, by Proposition \ref{sect-2-prop-1} and Lemma \ref{2-lem-5}, the two sides of \eqref{smooth-SDE-integral} converge respectively to $X$ and
  \begin{equation*}
  x+\int_0^{\textstyle\cdot} \sigma(X_s)\,\d B_s +\int_0^{\textstyle\cdot} b(X_s)\,\d s.
  \end{equation*}
Therefore, for almost all $x\in\R^d$, the following equality holds $\P$-almost surely:
  $$X_t(x)=x+\int_0^t \sigma(X_s)\,\d B_s +\int_0^t b(X_s)\,\d s,\quad\mbox{for all }t\in[0,T].$$
That is to say, $X_t$ solves SDE \eqref{Ito-SDE} over the time interval $[0,T]$. By Corollary \ref{sect-2-cor}, such a solution is also unique, thus we finish the proof of Theorem \ref{1-thm-0}.

\section{Proof of Theorem \ref{1-thm-1}}

We shall prove Theorem \ref{1-thm-1} by following the line of arguments in Section 3 and establishing some analogous estimates. We focus on the differences in this setting and will not provide detailed proofs for all the results.

We first give an a-priori estimate on the second moment of the flow. The main difference from Lemma \ref{sect-2-lem-1} is that we adopt here the Lebesgue measure as the reference measure, thus the integral is restricted on a finite ball $B_R$ and the averaged Radon--Nikodym density is assumed to be bounded in the space and time variables.

\begin{lemma}\label{2-lem-1}
Assume the conditions of Theorem \ref{1-thm-1}. Let $X_t$ be a generalized stochastic flow associated to It\^o SDE \eqref{Ito-SDE} and $K_t$ its Radon--Nikodym density. Assume that $\Lambda_T:=\sup\{\E K_t(x):(t,x)\in [0,T]\times \R^d\}<+\infty$. Then for any $R>0$
  $$\E\int_{B_R}\|X_\cdot (x)\|_{\infty,T}^2\,\d x \leq C_{T,r}\big((1+R^2)\L^d(B_R)+ T\Lambda_T \|b\|_{L^1 (B_r)}\big).$$
where the number $r$ comes from condition (ii) in Theorem \ref{1-thm-1} and the constant $C_{T,r}$ depends on the growth properties of $\sigma, b$.
\end{lemma}

\begin{proof}
We follow the idea in \emph{Step 2} of the proof of \cite[Theorem 2.2]{Zhang13}. For $\lambda>0$, let $\tau_\lambda(x)=\inf\{t>0: |X_t(x)|\geq \lambda\}$. We omit the space variable $x$ to simplify notations. By the It\^o formula, for any $t\leq T$,
  \begin{equation}\label{2-lem-1.1}
  \begin{split}
  \E\sup_{s\leq t}|X_{s\wedge \tau_\lambda}|^2&\leq |x|^2+2\,\E\sup_{s\leq t}\int_0^{s\wedge \tau_\lambda}\<X_u,\sigma(X_u)\,\d B_u\>\\
  &\hskip14pt + 2\,\E\sup_{s\leq t}\int_0^{s\wedge \tau_\lambda}\<X_u,b(X_u)\>\,\d u+ \E\sup_{s\leq t}\int_0^{s\wedge \tau_\lambda} |\sigma(X_u)|^2\,\d u\\
  &=: |x|^2+I_1(t)+I_2(t)+I_3(t).
  \end{split}
  \end{equation}
First, by Burkholder's inequality,
  \begin{align*}
  I_1(t)&\leq 4\,\E \bigg[\bigg(\int_0^{t\wedge \tau_\lambda} |\sigma(X_u)^\ast X_u|^2\,\d u\bigg)^{1/2}\bigg]\\
  &\leq 4\, \E \bigg[\sup_{u\leq t\wedge \tau_\lambda}|X_u| \cdot \bigg(\int_0^{t\wedge \tau_\lambda} |\sigma(X_u)|^2\,\d u\bigg)^{1/2}\bigg]\\
  &\leq \frac12\, \E\sup_{u\leq t\wedge \tau_\lambda}|X_u|^2  +8 \E \int_0^{t} |\sigma(X_{u\wedge \tau_\lambda})|^2\,\d u,
  \end{align*}
where the last step follows from Cauchy's inequality. Since $\sigma$ is bounded, we arrive at
  \begin{equation}\label{2-lem-1.2}
  I_1(t) \leq \frac12\, \E\sup_{u\leq t\wedge \tau_\lambda}|X_u|^2 +C_1t.
  \end{equation}
Similarly,
  \begin{equation}\label{2-lem-1.3}
  I_3(t) \leq C_1t.
  \end{equation}
Next,
  \begin{equation}\label{2-lem-1.4}
  \begin{split}
  I_2(t)&\leq 2\, \E \int_0^{t\wedge\tau_\lambda} |\<X_u, b(X_u)\>|\big(\ch_{\{|X_u|\leq r\}}+ \ch_{\{|X_u|> r\}}\big)\,\d u\\
  &\leq 2r \E \int_0^{t} |b(X_u)|\ch_{\{|X_u|\leq r\}}\,\d u +2 \E \int_0^{t} C_2(1+|X_{u\wedge\tau_\lambda}|^2)\,\d u.
  \end{split}
  \end{equation}
where in the last step we used the linear growth property of $b$ outside the ball $B_r$. Denote by
  $$\xi_t=\E\sup_{s\leq t}|X_{s\wedge \tau_\lambda}|^2=\E\sup_{s\leq t\wedge \tau_\lambda}|X_{s}|^2.$$
Combining \eqref{2-lem-1.1}--\eqref{2-lem-1.4} yields
  $$\xi_t\leq 2|x|^2+ 4r\E \int_0^{t} |b(X_u)|\ch_{\{|X_u|\leq r\}}\,\d u +C_3 \int_0^t (1+\xi_u)\,\d u.$$
Gronwall's inequality gives us
  $$\xi_T\leq C_{T,r}\bigg(1+|x|^2+\E \int_0^{T} |b(X_u)|\ch_{\{|X_u|\leq r\}}\,\d u\bigg).$$

By the definition of $\Lambda_T$, we have
  \begin{equation*}
  \begin{split}
  \int_{B_R}\E\sup_{s\leq T\wedge \tau_\lambda}|X_{s}|^2\,\d x&\leq C_{T,r}\bigg((1+R^2)\L^d(B_R)+ \int_0^{T} \E \int_{B_R}|b(X_u)|\ch_{\{|X_u|\leq r\}}\,\d x\d u\bigg)\\
  &\leq C_{T,r}\bigg((1+R^2)\L^d(B_R)+ \int_0^{T}\E \int_{B_r}|b(y)| K_u(y)\,\d y\d u\bigg)\\
  &\leq C_{T,r}\big((1+R^2) \L^d(B_R)+ T\Lambda_T \|b\|_{L^1(B_r)}\big).
  \end{split}
  \end{equation*}
Since the right hand side is independent of $\lambda$, Fatou's lemma yields the desired result.
\end{proof}

Next we shall establish in the current setting a stability estimate which is similar to Lemma \ref{sect-2-lem-2}. Recall the definition of $\psi_\delta$ in Section 2.

\begin{lemma}[Stability estimate]\label{2-lem-2}
Suppose that $\sigma,\tilde\sigma\in L^2_{loc}(\R^d, \mathcal{M}_{d\times m})$ and $b,\tilde b\in L^1_{loc}(\R^d, \R^d)$. Moreover, $\sigma$ and $b$ satisfy the assumptions in Theorem \ref{1-thm-1}. Let $X_t$ (resp. $\tilde X_t$) be the $\L^d$-a.e. stochastic flow associated to the It\^o SDE \eqref{Ito-SDE} with coefficients $\sigma$ and $b$ (resp. $\tilde\sigma$ and $\tilde b$). Denote by $K_t$ (resp. $\tilde K_t$) the Radon--Nikodym density of $X_t$ (resp. $\tilde X_t$) with respect to $\L^d$. Assume that
  \begin{equation}\label{2-lem-2.1}
  \Lambda_T:=\sup_{0\leq t\leq T,\, x\in\R^d}\big[\E K_t(x)\vee \E \tilde K_t(x)\big]<+\infty.
  \end{equation}
Then for any $\delta>0$,
  \begin{align*}
  \E\int_{B_R\cap G_\lambda}\psi_\delta\big(\|X-\tilde X\|^2_{\infty,T}\big)\,\d x  & \leq C_{d,R,T}\bigg[1+ \|g_{2\lambda}\|_{L^1(B_\lambda)} +\frac1{\delta}\|\sigma-\tilde\sigma\|_{L^2(B_\lambda)}^2\\
  &\hskip45pt +\frac1{\sqrt\delta}\big(\|\sigma-\tilde\sigma\|_{L^2(B_\lambda)}+ \|b-\tilde b\|_{L^1(B_\lambda)}\big)\bigg],
  \end{align*}
where $G_\lambda(\omega):=\big\{x\in \R^d:\|X_\cdot(\omega,x)\|_{\infty,T} \vee\|\tilde X_\cdot(\omega,x)\|_{\infty,T} \leq \lambda\big\}$ and $C_{d,R,T}$ depends on $d,R,T$ and $\Lambda_T$.
\end{lemma}

\begin{proof}
We follow the proof of Lemma \ref{sect-2-lem-2}. Denote by $Z_t=X_t-\tilde X_t$ and $\xi_t=|Z_t|^2$ where we omit the space variable $x$. By It\^o's formula, we still have
  \begin{equation}\label{2-lem-2.2}
  \begin{split}
  \d\psi_\delta(\xi_{t})&\leq 2\frac{\<Z_t,(\sigma(X_t)- \tilde\sigma(\tilde X_t))\,\d B_t\>}{\rho(\xi_t)+\delta} +2\frac{\<Z_t,b(X_t)-\tilde b(\tilde X_t)\>}{\rho(\xi_t)+\delta}\,\d t + \frac{\|\sigma(X_t)- \tilde\sigma(\tilde X_t)\|^2}{\rho(\xi_t)+\delta}\,\d t\\
  &=: \d I_1(t)+ \d I_2(t)+ \d I_3(t).
  \end{split}
  \end{equation}

We first estimate the term $I_1(t)$. For any $x\in B_R$ and $\lambda>0$, we define $\tau_\lambda(x)=\inf\{t>0: |X_t(x)|\vee |\tilde X_t(x)|>\lambda\}$. Note that $B_R\cap G_\lambda\subset \{x\in B_R:\tau_\lambda(x)>T\}$. Thus
  \begin{equation}\label{2-lem-2.3}
  \E\int_{B_R\cap G_\lambda}\sup_{0\leq t\leq T}|I_1(t)|\,\d x \leq \E\int_{B_R} \sup_{0\leq t\leq T\wedge\tau_\lambda} |I_1(t)|\,\d x.
  \end{equation}
By Burkholder's inequality, we have
  \begin{align*}
  \E\sup_{0\leq t\leq T\wedge\tau_\lambda} |I_1(t)|^2&\leq 4\, \E\int_0^{T\wedge\tau_\lambda} \frac{\big|(\sigma(X_s)-\tilde\sigma(\tilde X_s))^\ast Z_s\big|^2}{(\rho(\xi_s)+\delta)^2}\,\d s\\
  &\leq 4\, \E\int_0^{T\wedge\tau_\lambda} \frac{\|\sigma(X_s)- \tilde\sigma(\tilde X_s)\|^2}{\xi_s+\delta}\,\d s
  \end{align*}
since $\rho(\xi)\geq \xi\geq 0$. As a result,
  \begin{equation}\label{2-lem-2.4}
  \begin{split}
  \int_{B_R}\E \sup_{0\leq t\leq T\wedge\tau_\lambda} |I_1(t)| \,\d x&\leq \int_{B_R} 2\bigg[\E\int_0^{T\wedge\tau_\lambda} \frac{\|\sigma(X_s)- \tilde\sigma(\tilde X_s)\|^2} {\xi_s+\delta}\,\d s\bigg]^{1/2}\,\d x\\
  &\leq C_R \bigg[\int_{B_R} \E\int_0^{T\wedge\tau_\lambda} \frac{\|\sigma(X_s)- \tilde\sigma(\tilde X_s)\|^2}{\xi_s+\delta}\,\d s \d x\bigg]^{1/2}.
  \end{split}
  \end{equation}
where $C_R=2(\L_d(B_R))^{1/2}$. Since $\|\sigma(X_s)- \tilde\sigma(\tilde X_s)\|\leq \|\sigma(X_s)- \sigma(\tilde X_s)\|+\|\sigma(\tilde X_s)- \tilde\sigma(\tilde X_s)\|$ and $\sigma$ is globally Lipschitz continuous, we have
  \begin{align*}
  \int_{B_R}\E \sup_{0\leq t\leq T\wedge\tau_\lambda} |I_1(t)| \,\d x&\leq C_{R,T}+\sqrt{2}C_R \bigg[\int_{B_R} \E\int_0^{T\wedge\tau_\lambda} \frac{\|\sigma(\tilde X_s)- \tilde\sigma(\tilde X_s)\|^2} {\xi_s+\delta}\,\d s \d x\bigg]^{1/2}\\
  &\leq C_{R,T}+\sqrt{\frac2\delta}\, C_R \bigg[\int_0^T \E\int_{B_R}{\bf 1}_{\{\tau_\lambda>s\}} \|\sigma(\tilde X_s)- \tilde\sigma(\tilde X_s)\|^2\, \d x\d s\bigg]^{1/2},
  \end{align*}
where in the second step we have used the Fubini theorem. We have by \eqref{2-lem-2.1} that
  \begin{align*}
  \E\int_{B_R}{\bf 1}_{\{\tau_\lambda>s\}} \|\sigma(\tilde X_s)- \tilde\sigma(\tilde X_s)\|^2\, \d x&\leq \E\int_{B_\lambda} \|\sigma(y)- \tilde\sigma(y)\|^2\tilde K_s(y)\, \d y\\
  &\leq \Lambda_T \int_{B_\lambda} \|\sigma(y)- \tilde\sigma(y)\|^2 \, \d y.
  \end{align*}
Combining these arguments with \eqref{2-lem-2.3} gives us
  \begin{equation}\label{2-lem-2.5}
  \E\int_{B_R\cap G_\lambda}\sup_{0\leq t\leq T}|I_1(t)|\,\d x \leq C_{R,T}+\sqrt{\frac{2T \Lambda_T}\delta}\, C_R\|\sigma-\tilde\sigma\|_{L^2(B_\lambda)}.
  \end{equation}

Next, we treat the second term $I_2(t)$. We have
  \begin{align*}
  \E\int_{B_R\cap G_\lambda}\sup_{0\leq t\leq T}|I_2(t)|\,\d x& \leq \E\int_{B_R}\sup_{0\leq t\leq T\wedge\tau_\lambda}|I_2(t)|\,\d x\\
  &\leq 2\, \E\int_{B_R}\int_0^{T\wedge\tau_\lambda} \frac{|\<Z_s,b(X_s)-\tilde b(\tilde X_s)\>|}{\rho(\xi_s)+\delta}\,\d s\d x\\
  &=2 \int_0^T \E\int_{B_R}{\bf 1}_{\{\tau_\lambda>s\}} \frac{|\<Z_s,b(X_s)-\tilde b(\tilde X_s)\>|}{\rho(\xi_s)+\delta}\,\d x\d s.
  \end{align*}
Note that $|\<Z_s,b(X_s)-\tilde b(\tilde X_s)\>|\leq |\<Z_s,b(X_s)-b(\tilde X_s)\>|+|\<Z_s,b(\tilde X_s)-\tilde b(\tilde X_s)\>|$. We deduce from $(\mathbf{H}_1)$ that
  \begin{align*}
  \E\int_{B_R}{\bf 1}_{\{\tau_\lambda>s\}} \frac{|\<Z_s,b(X_s)-b(\tilde X_s)\>|}{\rho(\xi_s)+\delta}\,\d x &\leq \E\int_{B_R}{\bf 1}_{\{\tau_\lambda>s\}} \big[g_{2\lambda}(X_s) + g_{2\lambda}(\tilde X_s)\big]\,\d x\\
  &\leq \E\int_{B_\lambda} g_{2\lambda}(y) \big[K_s(y)+\tilde K_s(y)\big]\,\d y,
  \end{align*}
which, combined with \eqref{2-lem-2.1}, gives us
  $$ \int_0^T \E\int_{B_R}{\bf 1}_{\{\tau_\lambda>s\}} \frac{|\<Z_s,b(X_s)-b(\tilde X_s)\>|}{\rho(\xi_s)+\delta}\,\d x\d s\leq 2T\Lambda_T\|g_{2\lambda}\|_{L^1(B_\lambda)}.$$
Moreover, using again the fact that $\rho(\xi) \geq \xi \geq 0$, we obtain
  \begin{align*}
  \int_0^T \E\int_{B_R}{\bf 1}_{\{\tau_\lambda>s\}} \frac{|\<Z_s,b(\tilde X_s)-\tilde b(\tilde X_s)\>|}{\rho(\xi_s)+\delta}\,\d x\d s&\leq \int_0^T \E\int_{B_R}{\bf 1}_{\{\tau_\lambda>s\}} \frac{|b(\tilde X_s)-\tilde b(\tilde X_s)|}{\sqrt{\xi_s +\delta}}\,\d x\d s\\
  &\leq \frac1{\sqrt{\delta}} \int_0^T \E\int_{B_\lambda} |b(y)-\tilde b(y)| \tilde K_s(y)\,\d y\d s\\
  &\leq \frac{T\Lambda_T}{\sqrt{\delta}}\|b-\tilde b\|_{L^1(B_\lambda)}.
  \end{align*}
Consequently,
  \begin{equation}\label{2-lem-2.6}
  \begin{split}
  \E\int_{B_R\cap G_\lambda}\sup_{0\leq t\leq T}|I_2(t)|\,\d x &\leq 2T\Lambda_T\bigg(2\|g_{2\lambda}\|_{L^1(B_\lambda)} + \frac1{\sqrt{\delta}}\|b-\tilde b\|_{L^1(B_\lambda)}\bigg).
  \end{split}
  \end{equation}

Finally, the term $I_3(t)$ can be estimated as that on the right hand side of \eqref{2-lem-2.4}, thus we have
  \begin{align*}
  \E\int_{B_R\cap G_\lambda}\sup_{0\leq t\leq T}|I_3(t)|\,\d x &\leq \E\int_{B_R} \int_0^{T\wedge\tau_\lambda} \frac{\|\sigma(X_s)- \tilde\sigma(\tilde X_s)\|^2}{\rho(\xi_s)+\delta}\,\d s\d x\\
  &\leq C_{R,T}+\frac2\delta T\Lambda_T \|\sigma-\tilde\sigma\|_{L^2(B_\lambda)}^2.
  \end{align*}
Combining this estimate with \eqref{2-lem-2.5} and \eqref{2-lem-2.6}, we obtain the desired result.
\end{proof}

Now we start to prove the existence of generalized stochastic flow associated to \eqref{Ito-SDE}. In order to apply Lemma \ref{2-lem-2}, we need a uniform estimate on the  Radon--Nikodym densities. First, in the smooth case, we have (see \cite[Lemma 3.1]{Zhang10} for its proof)

\begin{lemma}\label{2-lem-3}
Suppose that $\sigma$ and $b$ are smooth with compact supports. Then for any $p\geq 1$, the Radon--Nikodym density $K_t:=\frac{\d(X_t)_\#\L^d}{\d\L^d}$ satisfies
  \begin{align*}
  &\hskip14pt \sup_{(t,x)\in[0,T]\times\R^d}\E K_t^p(x)\\
  &\leq \exp\bigg\{pT \bigg\|\Big(\frac p2 |\div(\sigma)|^2 -\div(b) +\sum_{k=1}^m \sum_{i,j=1}^d \Big[\frac12(\partial_i \sigma^{jk})(\partial_j \sigma^{ik}) + \sigma^{ik}\partial_{ij}\sigma^{jk}\Big] \Big)^+ \bigg\|_\infty\bigg\},
  \end{align*}
where $\div(\sigma)=(\div(\sigma^{1\cdot}),\cdots,\div(\sigma^{m\cdot}))$ are the divergences of the column vectors of $\sigma$.
\end{lemma}

Let $b_n$ be defined as in \eqref{2-convolution}. Moreover, $\sigma_n:=\sigma\,\phi_n$. Then for every $n\geq1$, the functions $\sigma_n$ and $b_n$ are smooth with compact supports. Consider the following It\^o's SDE:
  \begin{equation}\label{smooth-SDE}
  \d X^n_t=\sigma_n(X^n_t)\,\d B_t+b_n(X^n_t)\,\d t,\quad X^n_0=x
  \end{equation}
which gives rise to a stochastic flow $(X^n_t)_{0\leq t\leq T}$ of diffeomorphisms on $\R^d$. Denote by $K^n_t$ the Radon--Nikodym density of $(X^n_t)_\#\L^d$ with respect to $\L^d$. Then as in the proof of \cite[Theorem 2.6]{Zhang13}, we can show that

\begin{lemma}[Uniform density estimate]\label{2-lem-4}
Under the conditions of Theorem \ref{1-thm-1}, for any $p\geq 1$, one has
  \begin{align}\label{2-lem-4.1}
  \Lambda_{p,T}:=\sup_{n\geq 1}\sup_{(t,x)\in[0, T]\times\R^d}\E [K^n_t(x)]^p<+\infty.
  \end{align}
\end{lemma}

\begin{proof}
By Lemma \ref{2-lem-3}, we have
  \begin{equation}\label{2-lem-4.2}
  \begin{split}
  &\hskip14pt \sup_{(t,x)\in[0,T]\times\R^d}\E [K_t^n(x)]^p\\
  &\leq \exp\bigg\{pT \bigg\|\Big(\frac p2 |\div(\sigma_n)|^2 -\div(b_n) +\sum_{k=1}^m \sum_{i,j=1}^d \Big[\frac12(\partial_i \sigma_n^{jk})(\partial_j \sigma_n^{ik}) + \sigma_n^{ik}\partial_{ij}\sigma_n^{jk}\Big] \Big)^+ \bigg\|_\infty\bigg\}.
  \end{split}
  \end{equation}
Under the condition (i) of Theorem \ref{1-thm-1}, it is clear that
  \begin{equation}\label{2-lem-4.3}
  |\div(\sigma_n)|^2 + \sum_{k=1}^m \sum_{i,j=1}^d \Big(\frac12 \big|\partial_i \sigma_n^{jk})(\partial_j \sigma_n^{ik})\big| + \big|\sigma_n^{ik}\partial_{ij}\sigma_n^{jk}\big| \Big) \leq C_1 <\infty,
  \end{equation}
where $C_1$ is independent of $n$. Recall the definition of $\phi_n$ above \eqref{2-convolution}. We have
  $$|\nabla\phi_n(x)|\leq \frac{\|\nabla\phi\|_\infty}n \ch_{\{n\leq |x|\leq 2n\}} \leq \frac{3\|\nabla\phi\|_\infty}{1+|x|} \ch_{\{|x|\geq n\}}.$$
By \eqref{1-thm-1.2}, for all $n\geq r+1$,
  \begin{equation}\label{2-lem-4.4}
  \begin{split}
  -\div(b_n)&=-[\div(b)\ast\chi_n]\phi_n -\<b \ast\chi_n,\nabla\phi_n\>\\
  &\leq \big\|[\div(b)]^-\big\|_\infty + 3\|\nabla\phi\|_\infty\frac{|b \ast\chi_n|}{1+|x|}\ch_{\{|x|\geq n\}}.
  \end{split}
  \end{equation}
For any $y\in B_1$, one has $1+|x-y|\leq 2+|x|\leq 2(1+|x|)$, therefore, for $|x|\geq n\geq r+1$,
  $$\frac{|(b \ast\chi_n)(x)|}{1+|x|} \leq 2(|\bar b|\ast \chi_n)(x)\leq 2\|\bar b\|_{L^\infty(B_r^c)}.$$
Combining this estimate with \eqref{2-lem-4.2}--\eqref{2-lem-4.4}, we finish the proof.
\end{proof}

Using Lemmas \ref{2-lem-2} and \ref{2-lem-4}, we can now show that the sequence of stochastic flows generated by \eqref{smooth-SDE} strongly converges to a random field $X:\Omega \times \R^d\to C([0,T],\R^d)$.

\begin{proposition}\label{2-prop-1}
Assume the conditions of Theorem \ref{1-thm-1}. Then there exists a random field $X:\Omega\times\R^d\to C([0,T],\R^d)$ such that for any $p\in[1,2)$ and $R>0$,
  $$\lim_{n\to\infty}\E\int_{B_R} \|X^n-X\|_{\infty,T}^p\,\d x=0.$$
\end{proposition}

\begin{proof}
The proof is similar to that of Proposition \ref{sect-2-prop-1}. For any $n\geq1$, we denote by $G_\lambda^{n}$ the level set of the flow $X^n_t$ on the interval $[0,T]$:
  $$G_\lambda^n(\omega)=\{x\in \R^d:\|X^n_\cdot(\omega,x)\|_{\infty,T}\leq \lambda\}.$$
Applying Lemmas \ref{2-lem-1} and \ref{2-lem-4} to the flows $X^n_t$ gives us
  \begin{equation}\label{2-prop-1.0}
  \sup_{n\geq 1} \int_{B_R} \E\sup_{t\leq T}|X^n_t(x)|^2\,\d x\leq C<+\infty,
  \end{equation}
in which $C$ depends on $d,T,R, \Lambda_{1,T}$ and $\|b\|_{L^1(B_{r+1})}$. Therefore,
  \begin{equation}\label{2-prop-1.1}
  \E\big[\L^d\big(B_R\cap (G_\lambda^n\cap G_\lambda^l)^c\big)\big]\leq \E\big[\L^d\big(B_R\cap (G_\lambda^n)^c\big)\big]+\E\big[\L^d\big(B_R\cap (G_\lambda^l)^c\big)\big] \leq \frac{2C}{\lambda^2}.
  \end{equation}

Next, by Lemmas \ref{2-lem-2} and \ref{2-lem-4}, we have
  \begin{equation}\label{2-prop-1.2}
  \begin{split}
  &\hskip14pt \E\int_{B_R\cap G^n_\lambda\cap G^l_\lambda}\psi_\delta\big(\|X^n-X^l\|^2_{\infty,T}\big)\,\d x\\
  & \leq C_{d,R,T}\bigg[1+ \|g^n_{2\lambda}\|_{L^1(B_\lambda)} +\frac1\delta\|\sigma_n-\sigma_l\|_{L^2(B_\lambda)}^2\\
  &\hskip50pt +\frac1{\sqrt\delta}\big(\|\sigma_n-\sigma_l\|_{L^2(B_\lambda)} + \|b_n- b_l\|_{L^1(B_\lambda)}\big) \bigg],
  \end{split}
  \end{equation}
Now we define
  $$\delta_{n,l}= \big(\|\sigma_n-\sigma_l\|_{L^2(B_\lambda)} + \|b_n- b_l\|_{L^1(B_\lambda)}\big)^2$$
which tends to $0$ as $n,l\to+\infty$. Taking $\delta=\delta_{n,l}$ in \eqref{2-prop-1.2}, we obtain that for any $n,l\geq1$,
  \begin{equation}\label{2-prop-1.5}
  \begin{split}
  \E\int_{B_R\cap G^n_\lambda\cap G^l_\lambda}\psi_{\delta_{n,l}} \big(\|X^n-X^l\|^2_{\infty,T}\big)\,\d x
  &\leq C_{d,R,T}\big[3 + \|g_{2\lambda}\|_{L^1(B_{\lambda+1})}\big]=: C_{d,R,T,\lambda}< \infty.
  \end{split}
  \end{equation}

Next for $\eta\in (0,1)$, set
  $$\Theta^{n,l}_\eta=\big\{(\omega,x)\in\Omega\times\R^d:\|X^n-X^l\|_{\infty,T}\leq \eta\big\}.$$
We have by \eqref{2-prop-1.1} that
  \begin{equation}\label{2-prop-1.6}
  \begin{split}
  &\hskip14pt \E\int_{B_R}\big(1\wedge\|X^n-X^l\|_{\infty,T}^2 \big)\,\d x\\
  &\leq \E\big[\L^d\big(B_R\cap (G_\lambda^n\cap G_\lambda^l)^c\big)\big]
  +\E\int_{B_R\cap G_\lambda^n\cap G_\lambda^l}\big(1\wedge\|X^n-X^l\|_{\infty,T}^2 \big)\,\d x\\
  &\leq \frac{2C}{\lambda^2} +\eta^2 \L^d(B_R)+\frac{1}{\psi_{\delta_{n,l}}(\eta^2)} \E\int_{(B_R\cap G_\lambda^n\cap G_\lambda^l)\setminus \Theta^{n,l}_\eta} \psi_{\delta_{n,l}}(\|X^n-X^l\|_{\infty,T}^2)\, \d x\\
  &\leq \frac{2C}{\lambda^2} +\eta^2 \L^d(B_R)+\frac{C_{d,R,T,\lambda}}{\psi_{\delta_{n,l}}(\eta^2)},
  \end{split}
  \end{equation}
where the last inequality follows from \eqref{2-prop-1.5}. First letting $n,l\to +\infty$, and then $\lambda\to +\infty$, $\eta\to 0$, we obtain
  $$\lim_{n,l\to+\infty}\E\int_{B_R} \big(1\wedge\|X^n-X^l\|_{\infty,T}^2 \big)\,\d x=0.$$
Due to \eqref{2-prop-1.0}, we have for any $p\in[1,2)$,
  $$\lim_{n,l\to+\infty}\E\int_{B_R} \|X^n-X^l\|_{\infty,T}^p \,\d x=0.$$
Hence there exists a random field $X:\Omega\times\R^d\to C([0,T],\R^d)$ such that
  \begin{equation*}
  \lim_{n\to +\infty}\E\int_{B_R} \|X^n-X\|_{\infty,T}^p \,\d x=0.
  \end{equation*}
The proof is complete.
\end{proof}

\begin{remark}
From the above proof, we see that for any $p\in[1,2)$, the limit random field belongs to the space $L^p(\Omega\times B_R, C([0,T],\R^d))$ for all $R>0$.
\end{remark}

Finally we can present the

\begin{proof}[Proof of Theorem \ref{1-thm-1}]
In view of Lemma \ref{2-lem-4} and Proposition \ref{2-prop-1}, we can directly apply \cite[Lemma 3.4]{Zhang13} to show that for all $t\in[0,T]$, there exists $K_t:\Omega\times\R^d\to\R_+$ such that $(X_t)_\#\L^d=K_t \L^d$; moreover, $\sup_{(t,x)\in [0,T]\times\R^d}\E K_t(x)<\infty$. By the same arguments as those at the end of Section 3, we can prove that $(X_t)_{0\leq t\leq T}$ constructed in Proposition \ref{2-prop-1} is an $\L^d$-a.e. stochastic flow associated to the It\^o SDE \eqref{Ito-SDE}.

The uniqueness of $\L^d$-a.e. stochastic flows follow from Lemma \ref{2-lem-2}. Indeed, let $(X_t)_{0\leq t\leq T}$ and $(\tilde X_t)_{0\leq t\leq T}$ be two generalized flows associated to the It\^o SDE \eqref{Ito-SDE}. For any $\lambda>R>0$, applying Lemma \ref{2-lem-2} with $\tilde\sigma=\sigma$ and $\tilde b=b$ leads to
  \begin{equation}\label{2-cor-1}
  \E\int_{B_R\cap G_\lambda}\psi_\delta\big(\|X-\tilde X\|^2_{\infty,T}\big)\,\d x \leq C_{d,T,R}\big(1+\|g_{2\lambda}\|_{L^1(B_\lambda)}\big)=: C_{d,T,R,\lambda} .
  \end{equation}
For any $\eta\in (0,1)$, let $\Theta_\eta(\omega):=\{x\in B_R: \|X_\cdot(\omega,x)- \tilde X_\cdot(\omega,x)\|_{\infty,T}\leq \eta\}$. Then, analogous to \eqref{2-prop-1.6}, we can deduce from Lemma \ref{2-lem-1} and \eqref{2-cor-1} that
  \begin{align*}
  \E\int_{B_R}\big(1\wedge \|X-\tilde X\|_{\infty,T}^2 \big)\,\d x \leq \frac{C_{T,R,r}}{\lambda^2} +\eta^2 \L^d(B_R) + \frac{C_{d,T,R,\lambda}}{\psi_\delta(\eta^2)}.
  \end{align*}
We first let $\delta\downarrow 0$, then let $\lambda\uparrow \infty$ and $\eta\downarrow 0$ to get
  $$\E\int_{B_R}\big(1\wedge \|X-\tilde X\|_{\infty,T}^2 \big)\,\d x =0,$$
which yields the uniqueness since $R>0$ is arbitrary.
\end{proof}

\section{Uniqueness of Fokker--Planck equations}

This last section is devoted to the proof of Theorem \ref{4-thm}. We first give some preparations which are mainly taken from \cite[Section 2]{Luo14}, see also the beginning parts of \cite[Sections 1 and 2]{RocknerZhang10}. Denote by $\mathbb W_T^n= C([0,T],\R^n)$ the space of continuous functions from $[0,T]$ to $\R^n$. Let $\F_t^n$ be the canonical filtration generated by coordinate process $W_t(w)=w_t,\, w\in \mathbb W_T^n$. Recall that $\mathcal P(\R^d)$ is the set of probability measures on $(\R^d,\mathcal B(\R^d))$. To fix the notations, we  state in detail the two well known notions of solutions to \eqref{Ito-SDE}.

\begin{definition}[Martingale solution]
Given $\mu_0\in \mathcal P(\R^d)$, a probability measure $P_{\mu_0}$ on $(\mathbb W_T^d,\F_T^d)$ is called a martingale solution to SDE \eqref{Ito-SDE} with initial distribution $\mu_0$ if $P_{\mu_0}\circ w_0^{-1} =\mu_0$, and for any $\varphi\in C_c^\infty(\R^d)$, $\varphi(w_t)-\varphi(w_0)-\int_0^t L \varphi(w_s)\,\d s$ is an $(\F_t^d)$-martingale under $P_{\mu_0}$, where $L$ is the second order differential operator given in \eqref{sect-4.1}.
\end{definition}

\begin{definition}[Weak solution]\label{sect-2-def-2}
Let $\mu_0\in\mathcal P(\R^d)$. The SDE \eqref{Ito-SDE} is said to have a weak solution with initial law $\mu_0$ if there exist a filtered probability space $(\Omega,\mathcal G,(\mathcal G_t)_{0\leq t\leq T},P)$, on which are defined a $(\mathcal G_t)$-adapted continuous process $X_t$ taking values in $\R^d$ and an $m$-dimensional standard $(\mathcal G_t)$-Brownian motion $W_t$, such that $X_0$ is distributed as $\mu_0$ and a.s.,
  $$X_t=X_0+\int_0^t \sigma(X_s)\,\d W_s+\int_0^t b(X_s)\,\d s,\quad \forall\, t\in[0,T].$$
We denote this solution by $\big(\Omega,\mathcal G,(\mathcal G_t)_{0\leq t\leq T},P;X,W\big)$.
\end{definition}

The next result can be found in the proof of \cite[Chap. IV, Theorem 1.1]{Ikeda}.

\begin{proposition}\label{sect-5-prop-1}
Given two weak solutions $\big(\Omega^{(i)},\mathcal G^{(i)},\big(\mathcal G_t^{(i)}\big)_{0\leq t\leq T},P^{(i)}; X^{(i)},W^{(i)}\big),\,i=1,2$ to SDE \eqref{Ito-SDE}, having the same initial law $\mu_0\in\mathcal P(\R^d)$, there exist a filtered probability space $(\Omega,\mathcal G,(\mathcal G_t)_{0\leq t\leq T},P)$, a standard $m$-dimensional $(\mathcal G_t)$-Brownian motion $W_t$ and two $\R^d$-valued $(\mathcal G_t)$-adapted continuous processes $Y^{(i)},\,i=1,2$, such that $P\big(Y^{(1)}_0=Y^{(2)}_0\big)=1$ and for $i=1,2$, $X^{(i)}$ and $Y^{(i)}$ have the same distributions in $\mathbb W_T^d$, and $\big(\Omega,\mathcal G, (\mathcal G_t)_{0\leq t\leq T},P;Y^{(i)},W\big)$ is a weak solution of SDE \eqref{Ito-SDE}.
\end{proposition}

The assertion below is a special case of \cite[Chap. IV, Proposition 2.1]{Ikeda}.

\begin{proposition}[Existence of martingale solution implies that of weak solution]\label{sect-5-prop-2}
Let $\mu_0\in\mathcal P(\R^d)$ and $P_{\mu_0}$  be a martingale solution of SDE \eqref{Ito-SDE}. Then there exists a weak solution $(\Omega,\mathcal G,(\mathcal G_t)_{0\leq t\leq T},P; X,W)$ to SDE \eqref{Ito-SDE} such that $P\circ X^{-1}=P_{\mu_0}$.
\end{proposition}

Finally we recall the following result which is a consequence of Figalli's representation theorem (see \cite[Theorem 2.6]{Figalli}) for solutions to the Fokker--Planck equation \eqref{FPE}.

\begin{proposition}\label{sect-5-prop-Figalli}
Assume that $\sigma$ and $b$ are two bounded measurable functions. Given $\mu_0\in\mathcal P(\R^d)$, let $\mu_t\in\mathcal P(\R^d)$ be a measure-valued solution to equation \eqref{FPE} with initial value $\mu_0$. Then there exists a martingale solution $P_{\mu_0}$ to SDE \eqref{Ito-SDE} with initial law $\mu_0$ such that for all $\varphi\in C_c^\infty(\R^d)$, one has
  $$\int_{\R^d} \varphi(x)\,\d\mu_t(x)=\int_{\mathbb W_T^d}\varphi(w_t)\,\d P_{\mu_0}(w),   \quad \forall\, t\in[0,T].$$
\end{proposition}

Now we are ready to present

\begin{proof}[Proof of Theorem \ref{4-thm}]
We follow the idea of the proof of \cite[Theorem 1.2]{Luo14} (see also \cite[Theorem 1.1]{RocknerZhang10}). Let $u^{(i)}_t,i=1,2$ be two weak solutions to \eqref{FPE-1} in the class $L^\infty\big([0,T],L^1\cap L^\infty(\R^d)\big)$ with the same initial value $f$. Set $\d \mu_0(x)=f(x)\,\d x$. By Proposition \ref{sect-5-prop-Figalli}, there exist two martingale solutions $P^{(i)}_{\mu_0},i=1,2$ to the SDE \eqref{Ito-SDE} with the same initial probability distribution $\mu_0$, such that for all $\varphi\in C_c^\infty(\R^d)$,
  \begin{equation}\label{proof.0}
  \int_{\R^d}\varphi(x)u^{(i)}_t(x)\,\d x=\int_{C([0,T],\R^d)}\varphi(w_t)\,\d P^{(i)}_{\mu_0}(w),\quad i=1,2.
  \end{equation}
Proposition \ref{sect-5-prop-2} implies that there are two weak solutions $\big(\Omega^{(i)},\mathcal{G}^{(i)}, \big(\mathcal{G}_t^{(i)}\big)_{0\leq t\leq T},P^{(i)}; X^{(i)}, W^{(i)}\big)$, $ i=1,2$ to SDE \eqref{Ito-SDE} satisfying $P^{(i)}\circ \big(X^{(i)}\big)^{-1}=P^{(i)}_{\mu_0},\,i=1,2$. Finally, by Proposition \ref{sect-5-prop-1}, we can find a common filtered probability space $(\Omega,\mathcal{G},(\mathcal{G}_t)_{0\leq t\leq T},P)$, on which are defined a standard $m$-dimensional $(\mathcal{G}_t)$-Brownian motion $W$ and two continuous $(\mathcal{G}_t)$-adapted processes $Y^{(i)}\, (i=1,2)$, such that $P\big(Y^{(1)}_0=Y^{(2)}_0\big)=1$ and $Y^{(i)}$ is distributed as $P^{(i)}_{\mu_0}$ on $C([0,T],\R^d)$; moreover for $i=1,2$, it holds a.s. that
  $$Y^{(i)}_t=Y^{(i)}_0+\int_0^t b\big(Y^{(i)}_s\big)\,\d s +\int_0^t \sigma\big(Y^{(i)}_s\big)\,\d W_s\quad \mbox{for all } t\leq T.$$

Set $Z_t=Y^{(1)}_t-Y^{(2)}_t$ and for $\lambda>0$, define the stopping time $\tau_\lambda=\inf\big\{t\in[0,T]: \big|Y^{(1)}_t\big|\vee \big|Y^{(2)}_t\big|\geq \lambda\big\}$ with the convention that $\inf\emptyset=T$. Since the coefficients $\sigma$ and $b$ are bounded, it is clear that
  \begin{equation}\label{proof.1}
  \lim_{\lambda\to\infty} \tau_\lambda(\omega)=T\quad \mbox{almost surely}.
  \end{equation}
Fix $\delta>0$. We have by It\^o's formula that
  \begin{equation*}
  \begin{split}
  \psi_\delta \big(|Z_{t\wedge \tau_\lambda}|^2\big)
  &=\int_0^{t\wedge \tau_\lambda}\frac{2\big\langle Z_s,b\big(Y^{(1)}_s\big)-b\big(Y^{(2)}_s\big)\big\rangle
  +\big\|\sigma\big(Y^{(1)}_s\big)-\sigma\big(Y^{(2)}_s\big)\big\|^2}{\rho(|Z_s|^2)+\delta}\,\d s\\
  &\hskip13pt +2\int_0^{t\wedge \tau_\lambda}\frac{\big\langle Z_s,\big[\sigma\big(Y^{(1)}_s\big) -\sigma\big(Y^{(2)}_s\big)\big]\,\d W_s\big\rangle} {\rho(|Z_s|^2)+\delta}\\
  &\hskip13pt -2\int_0^{t\wedge \tau_\lambda} \rho'(|Z_s|^2)\frac{\big\|\big[\sigma\big(Y^{(1)}_s\big) -\sigma\big(Y^{(2)}_s\big)\big]^\ast Z_s\big\|^2} {({\rho(|Z_s|^2)+\delta})^2}\,\d s.
  \end{split}
  \end{equation*}
Since $\rho'\geq 0$, taking expectation on both sides with respect to $P$ yields
  \begin{equation}\label{proof.2}
  \begin{split}
  \E\psi_\delta \big(|Z_{t\wedge \tau_\lambda}|^2\big)
  &\leq\E \int_0^{t\wedge \tau_\lambda} \frac{\big\|\sigma\big(Y^{(1)}_s\big) -\sigma\big(Y^{(2)}_s\big)\big\|^2} {\rho(|Z_s|^2)+\delta}\,\d s +2\,\E\int_0^{t\wedge \tau_\lambda}\frac{\big\langle Z_s,b\big(Y^{(1)}_s\big)-b\big(Y^{(2)}_s\big)\big\rangle} {\rho(|Z_s|^2)+\delta}\,\d s\\
  & =: I_1+I_2.
  \end{split}
  \end{equation}
We shall estimate the two terms in the next two steps, respectively.

\emph{Step 1.} We first deal with the term $I_1$. Let $\sigma_n$ be defined as in \eqref{2-convolution}. By the triangular inequality, for any $n\geq 1$, we have
  \begin{equation}\label{proof.3}
  \begin{split}
  I_1&\leq 3\,\E \int_0^{t\wedge \tau_\lambda} \frac{\big\|\sigma_n\big(Y^{(1)}_s\big) -\sigma_n\big(Y^{(2)}_s\big)\big\|^2} {\rho(|Z_s|^2)+\delta}\,\d s\\
  &\hskip13pt +3\,\E \int_0^{t\wedge \tau_\lambda} \frac{\big\|\sigma_n\big(Y^{(1)}_s\big) -\sigma\big(Y^{(1)}_s\big)\big\|^2 +\big\|\sigma_n\big(Y^{(2)}_s\big)
  -\sigma\big(Y^{(2)}_s\big)\big\|^2} {\rho(|Z_s|^2)+\delta}\,\d s\\
  &=: I_{1,1}+I_{1,2}.
  \end{split}
  \end{equation}
To estimate the first term $I_{1,1}$, we shall use the hypothesis $(\mathbf{H}_\sigma)$. Fix any $x,y\in B_\lambda$, when $n>\lambda$, we have
  $$\|\sigma_n(x)-\sigma_n(y)\| = \|(\sigma\ast\chi_n)(x)- (\sigma\ast\chi_n)(y)\|\leq \int_{\R^d} \|\sigma(x-z) -\sigma(y-z)\|\chi_n(z)\,\d z.$$
Note that $(x-N)\cup (y-N)$ is a negligible set. For any $z\notin (x-N)\cup (y-N)$, one has $x-z\notin N$ and $y-z\notin N$. Thus by Cauchy's inequality and $(\mathbf{H}_\sigma)$,
  \begin{equation}\label{proof.3.5}
  \begin{split} 
  \|\sigma_n(x)-\sigma_n(y)\|^2&\leq \int_{\R^d} \|\sigma(x-z) -\sigma(y-z)\|^2 \chi_n(z)\,\d z\\
  &\leq \int_{\R^d} \big(\tilde g_{2\lambda}(x-z) + \tilde g_{2\lambda}(y-z)\big) \rho(|x-y|^2) \chi_n(z)\,\d z\\
  &= \big(\tilde g_{2\lambda}^n(x) + \tilde g_{2\lambda}^n(y)\big) \rho(|x-y|^2), 
  \end{split}
  \end{equation}
where $\tilde g_{2\lambda}^n =\tilde g_{2\lambda}\ast\chi_n$ is the convolution. Thus, if $n>\lambda$, then
  \begin{align*}
  I_{1,1}&\leq 3\,\E \int_0^{t\wedge \tau_\lambda} \big[ \tilde g_{2\lambda}^n\big(Y^{(1)}_s\big) +\tilde g_{2\lambda}^n\big(Y^{(2)}_s\big)\big] \,\d s\\
  &\leq 3\,\E \int_0^t \Big[\tilde g_{2\lambda}^n \big(Y^{(1)}_s\big) \ch_{\{|Y^{(1)}_s|\leq \lambda\}}+ \tilde g_{2\lambda}^n\big(Y^{(2)}_s\big) \ch_{\{|Y^{(2)}_s|\leq \lambda\}}\Big] \d s.
  \end{align*}
Recall that $Y^{(i)}_s$ has the same law with $X^{(i)}_s$, which is distributed as $u^{(i)}_s(x)\,\d x,\,i=1,2$. Consequently,
  \begin{align*}
  I_{1,1}&\leq 3\int_0^t \int_{B_\lambda} \tilde g_{2\lambda}^n(x) \big(u^{(1)}_s(x) +u^{(2)}_s(x)\big)\,\d x \d s\\
  &\leq 3\sum_{i=1}^2\big\|u^{(i)}\big\|_{L^\infty([0,T],L^\infty(\R^d))} \int_0^t \int_{B_\lambda} \tilde g_{2\lambda}^n(x) \,\d x \d s\\
  &\leq \tilde C T \|\tilde g_{2\lambda}\|_{L^1(B_{\lambda+1})}.
  \end{align*}
Note that the bound is independent of $n\geq 1$. In the same way,
  \begin{align*}
  I_{1,2}&\leq \frac 3{\delta}\sum_{i=1}^2\E\int_0^t \big\|\sigma_n\big(Y^{(i)}_s\big) -\sigma\big(Y^{(i)}_s\big)\big\|^2\ch_{\{|Y^{(i)}_s|\leq \lambda\}}\,\d s\\
  &\leq \frac 3{\delta}\sum_{i=1}^2 \int_0^t\int_{B_\lambda}\|\sigma_n(x) -\sigma(x)\|^2 u^{(i)}_s(x)\,\d x \d s\\
  &\leq \frac {3T}{\delta}\sum_{i=1}^2 \big\|u^{(i)}\big\|_{L^\infty([0,T],L^\infty(\R^d))} \int_{B_\lambda}\|\sigma_n(x)-\sigma(x)\|^2\,\d x
  \end{align*}
which vanishes as $n\to \infty$, since $\sigma\in L^{2}_{loc}(\R^d)$. Combining the above two estimates and letting $n\to\infty$ on the right hand side of \eqref{proof.3}, we obtain
  \begin{equation}\label{proof.4}
  I_1\leq \tilde C T \|\tilde g_{2\lambda}\|_{L^1(B_{\lambda+1})}=:\tilde C_{T,\lambda} <+\infty.
  \end{equation}

\emph{Step 2.} The estimate of the second term $I_2$ is analogous to that of $I_1$. Recall the definition of $b_n$ in \eqref{2-convolution}. Then similar to \eqref{proof.3}, we have
  \begin{equation}\label{proof.5}
  \begin{split}
  I_2 &= 2\,\E\int_0^{t\wedge\tau_\lambda}\frac{\big\<Z_s, b_n\big(Y^{(1)}_s\big)-b_n\big(Y^{(2)}_s\big)\big\>} {\rho(|Z_s|^2)+\delta}\,\d s\\
  &\hskip13pt + 2\,\E\int_0^{t\wedge\tau_\lambda}\frac{\big\<Z_s,b\big(Y^{(1)}_s\big) -b_n\big(Y^{(1)}_s\big)\big\>
  +\big\<Z_s,b_n\big(Y^{(2)}_s\big) - b\big(Y^{(2)}_s\big)\big\>}{\rho(|Z_s|^2)+\delta}\,\d s\\
  &=: I_{2,1}+I_{2,2}.
  \end{split}
  \end{equation}
The estimate of the term $I_{2,2}$ is analogous to that of $I_{1,2}$: since $\rho(s)\geq s\geq 0$,
  \begin{align*}
  I_{2,2} &\leq 2\sum_{i=1}^2 \E\int_0^{t\wedge \tau_\lambda} \frac{\big|b_n\big(Y^{(i)}_s\big) -b\big(Y^{(i)}_s\big)\big|}{\sqrt{|Z_s|^2+\delta}} \,\d s\\
  &\leq \frac 2{\sqrt\delta}\sum_{i=1}^2\E\int_0^t\big|b_n\big(Y^{(i)}_s\big) -b\big(Y^{(i)}_s\big)\big|\ch_{\{|Y^{(i)}_s|\leq \lambda\}}\,\d s\\
  &\leq \frac 2{\sqrt\delta} \sum_{i=1}^2 \int_0^t\int_{B_\lambda}|b_n(x)-b(x)|u^{(i)}_s(x)\,\d x \d s.
  \end{align*}
Since $b\in L^\infty(\R^d)$,
  \begin{equation}\label{proof.6}
  \begin{split}
  I_{2,2}&\leq \frac {2T}{\sqrt\delta} \sum_{i=1}^2 \big\|u^{(i)} \big\|_{L^\infty([0,T],L^\infty(\R^d))} \int_{B_\lambda}|b_n(x)-b(x)|\,\d x \to 0 \quad \mbox{as } n\to \infty.
  \end{split}
  \end{equation}

Finally we deal with the term $I_{2,1}$. By \eqref{Osgood-Sobolev.1}, similar computations as in \eqref{proof.3.5} lead to
  $$ |\<x-y, b_n(x)-b_n(y)\>|\leq \big(g^n_{2\lambda}(x)+g^n_{2\lambda}(y)\big)\rho(|x-y|^2) \quad \mbox{for all } x,y\in B_\lambda \mbox{ and } n>\lambda.$$
Therefore, for any $n>\lambda$,
  \begin{align*}
  I_{2,1}&\leq 2\,\E \int_0^{t\wedge\tau_\lambda}\big[g^n_{2\lambda}\big(Y^{(1)}_s\big)+ g^n_{2\lambda}\big(Y^{(2)}_s\big)\big]\,\d s\\
  &\leq 2\,\E \int_0^{t}\Big[ g^n_{2\lambda}\big( Y^{(1)}_s\big) \ch_{\{|Y^{(1)}_s|\leq \lambda\}}+ g^n_{2\lambda}\big(Y^{(2)}_s\big)\ch_{\{|Y^{(2)}_s|\leq \lambda\}}\Big]\,\d s.
  \end{align*}
Then, analogous to the above calculations,
  \begin{align*}
  I_{2,1}&\leq 2\sum_{i=1}^2 \int_0^{t}\! \int_{B_\lambda} g^n_{2\lambda}(x) u^{(i)}_s(x)\,\d x\d s\\
  &\leq 2T \sum_{i=1}^2 \big\|u^{(i)} \big\|_{L^\infty([0,T],L^\infty(\R^d))} \|g_{2\lambda}\|_{L^1(B_{\lambda+1})} =: \hat C_{T,\lambda}<+\infty.
  \end{align*}
This estimate together with \eqref{proof.5} and \eqref{proof.6} yields
  \begin{equation}\label{proof.7}
  I_2\leq \hat C_{T,\lambda}<+\infty.
  \end{equation}

\emph{Step 3.} Combining \eqref{proof.2}, \eqref{proof.4} and \eqref{proof.7}, we obtain
  $$\E\psi_\delta \big(|Z_{t\wedge \tau_\lambda}|^2\big)\leq \bar C_{T,\lambda}<+\infty,$$
where $\bar C_{T,\lambda}=\tilde C_{T,\lambda}+\hat C_{T,\lambda}$. Fix any $\eta>0$. The above inequality implies
  \begin{align*}
  P\big(|Z_{t\wedge \tau_\lambda}|>\eta\big) &\leq \frac{\E \psi_\delta\big(|Z_{t\wedge \tau_\lambda}|^2\big)}{\psi_\delta (\eta^2)}\leq \frac{\bar C_{T,\lambda}}{\psi_\delta (\eta^2)}.
  \end{align*}
Letting $\delta\downarrow0$, we arrive at $P\big(|Z_{t\wedge \tau_\lambda}|>\eta\big)=0$. Since $\eta$ can be arbitrarily small, it follows that $Z_{t\wedge \tau_\lambda}=0$ almost surely. Finally, we conclude from \eqref{proof.1} that for any $t\in[0,T]$, $Z_t=Y^{(1)}_t-Y^{(2)}_t=0$ a.s. The continuity of the two processes $Y^{(1)}_t$ and $Y^{(2)}_t$ yields that, almost surely, $Y^{(1)}_t=Y^{(2)}_t$ for all $t\in[0,T]$. Therefore $P^{(1)}_{\mu_0}=P^{(2)}_{\mu_0}$, which, together with the representation formula \eqref{proof.0}, leads to the uniqueness of solutions to \eqref{FPE-1}.
\end{proof}

\end{document}